\documentclass[10pt]{amsart}
\usepackage{latexsym}
\usepackage{amsmath,amssymb,amsthm,ascmac}
\usepackage{wrapfig}
\usepackage[all]{xy}
\usepackage{proof}
\usepackage{graphicx,xcolor}
\usepackage{pifont}

\numberwithin{equation}{section}

\theoremstyle{plain}
\newtheorem{theorem}{Theorem}[section]

\newtheorem{lemma}[theorem]{Lemma}

\newtheorem{cor}[theorem]{Corollary}
\newtheorem{prop}[theorem]{Proposition}

\newtheorem{fact}{Fact}
\newtheorem{obs}[theorem]{Observation}

\theoremstyle{definition}
\newtheorem{definition}[theorem]{Definition}
\newtheorem{example}[theorem]{Example}

\newtheorem*{remark}{Remark}
\newtheorem*{claim}{Claim}

\newtheorem*{notation}{Notation}
\newtheorem*{ack}{Acknowledgements}

\DeclareFontFamily{U}{dmjhira}{}
\DeclareFontShape{U}{dmjhira}{m}{n}{ <-> dmjhira }{}

\DeclareRobustCommand{\yo}{\text{\usefont{U}{dmjhira}{m}{n}\symbol{"48}}}
\DeclareMathOperator*{\colim}{colim}

\definecolor{qq}{rgb}{.0,.5,.5}

\newcommand{\N}{\mathbb{N}}
\newcommand{\om}{\omega}

\newcommand{\A}{\mathcal{A}}

\newcommand{\C}{\mathcal{C}}
\newcommand{\D}{\mathcal{D}}

\newcommand{\K}{\mathcal{K}}

\newcommand{\R}{\mathbb{R}}

\newcommand{\ittm}{{\ding{52}}}
\newcommand{\dame}{{\ding{56}}}

\newcommand{\at}[1]{|_{#1}}

\newcommand{\pair}[1]{\langle #1 \rangle}

\newcommand{\inj}{\rightarrowtail}
\newcommand{\surj}{\twoheadrightarrow}

\newcommand{\mono}{\rightarrowtail}
\newcommand{\embed}{\hookrightarrow}
\newcommand{\leftembed}{\hookleftarrow}

\newcommand{\edge}[1]{\overset{#1}{\longrightarrow}}

\newcommand{\leftedge}[1]{\overset{#1}{\longleftarrow}}

\newcommand{\monoedge}[1]{\overset{#1}{\inj}}
\newcommand{\embedge}[1]{\overset{#1}{\embed}}
\newcommand{\xedge}[1]{\xrightarrow{#1}}

\newcommand{\acts}{\curvearrowright}

\newdir{ >}{{}*!/-10pt/@{>}} 

\title{Atomic toposes, choice spectra and model-theoretic stability}
\author{Takayuki Kihara}
\begin{document}

\maketitle

\begin{abstract}
For a suitable class of structures $\K$, let ${\rm Sh}(\K_\lambda^{\rm op},J_{\rm at})$ denote the atomic sheaf topos over the opposite of the category $\K_\lambda$ of structures in $\K$ of cardinality at most $\lambda$, with embeddings as morphisms.
In this article, we analyze the relationship between model-theoretic properties of $\K$ and logical properties of ${\rm Sh}(\K_\lambda^{\rm op},J_{\rm at})$.

For an infinite cardinal $\lambda$, we show that $\K$ is Galois $\lambda$-stable if and only if the sheaf topos ${\rm Sh}(\K_\lambda^{\rm op},J_{\rm at})$ satisfies the axiom of choice for $\lambda^+$-indexed families.
We further show that every structure in $\K_\lambda$ extends to a maximal structure if and only if ${\rm Sh}(\K_\lambda^{\rm op},J_{\rm at})$ satisfies the internal axiom of choice.
If $\K_{\rm fin}$ is a class of finite structures, the cardinalities of the structures in each connected component of $\K_{\rm fin}$ have a finite upper bound if and only if ${\rm Sh}(\K^{\rm op}_{\rm fin},J_{\rm at})$ satisfies the internal axiom of (countable) choice.

Furthermore, via the topological Galois representation, we establish a correspondence between the criteria for the axiom of choice for atomic sheaf toposes and those for permutation models in set theory.
\end{abstract}

\section{Introduction}\label{sec:Introduction}
Sheaves for the double-negation topology have long been a subject of interest to logicians.
Lawvere \cite{Law71} and Tierney \cite{Tie72} first related Cohen forcing to the $\neg\neg$-sheaf topos ${\rm Sh}(\mathbb{P},\neg\neg)$ over a poset $\mathbb{P}$.
In general, one may introduce a sheaf topos ${\rm Sh}(\C,\neg\neg)$ not only for a poset $\mathbb{P}$ but also for a small category $\C$, so it seems interesting to investigate how properties of the category $\C$ are linked to logical properties of the topos ${\rm Sh}(\C,\neg\neg)$.


In particular, for a class $\K$ of (first-order) structures, it is natural to study the sheaf topos ${\rm Sh}(\K_\lambda^{\rm op},\neg\neg)$ over the opposite of the category $\K_\lambda$ of structures in $\K$ of cardinality at most $\lambda$, with embeddings as morphisms.
If $\mathcal{K}$ has the amalgamation property, then the double-negation topology is indeed an atomic topology $J_{\rm at}$, so the corresponding sheaf topos ${\rm Sh}(\K_\lambda^{\rm op},J_{\rm at})$ is called an atomic topos.
Atomic toposes associated with categories of models have been studied from a model-theoretic perspective by Caramello, in connection with countable categoricity and Fra\"{i}ss\'{e} constructions \cite{Car12,Car14}.

The subject of this article is the striking relationship between model-theoretic properties of $\K$ and logical properties of ${\rm Sh}(\K_\lambda^{\rm op},J_{\rm at})$.
Our main results state that, for a sufficiently well-behaved class $\K$ of structures and an (external) infinite cardinal $\lambda$:
\begin{itemize}
\item Theorem \ref{thm:main-stability-thm}: $\K$ is Galois $\lambda$-stable if and only if ${\rm Sh}(\K_\lambda^{\rm op},J_{\rm at})$ satisfies the axiom of choice for $\lambda^+$-indexed families.
\item Theorem \ref{thm:IAC-maximalcofinal}: Every structure in $\K_\lambda$ extends to a maximal structure if and only if ${\rm Sh}(\K_\lambda^{\rm op},J_{\rm at})$ satisfies the internal axiom of choice.
\item Theorem \ref{thm:EAC-main-thm}: In addition, every maximal structure in such a $\K_\lambda$ is rigid if and only if ${\rm Sh}(\K_\lambda^{\rm op},J_{\rm at})$ satisfies the external axiom of choice.
\item Theorem \ref{thm:main-finite-structures}: If $\K_{\rm fin}$ is a class of finite structures, the cardinalities of the structures in each connected component of $\K_{\rm fin}$ have a finite upper bound if and only if ${\rm Sh}(\K^{\rm op}_{\rm fin},J_{\rm at})$ satisfies the internal axiom of (countable) choice.
\end{itemize}

Examples for concrete classes of structures are summarized in Table \ref{table-1} (for an infinite cardinal $\lambda$).
Readers familiar with model-theoretic stability theory may infer from Table \ref{table-1} that type spectra and stability play a central role.
More precisely, the relevant notion is that of Galois $1$-types for embeddings (pointed extensions).

\begin{table}[t]\centering\small
\begin{tabular}{c|cccc}
Class $\mathcal{K}$ & $\lambda$-indexed  &  $\lambda^+$-indexed & Internal & Internal \\
of structures & choice & choice & dependent choice & choice \\ \hline
Sets & \ittm & \ittm & \ittm &\dame  \\
Vector spaces & \ittm & \ittm & \ittm& \dame \\
Groups & \ittm &  \dame & \ittm &\dame \\
Abelian groups & \ittm & \ittm & \ittm &\dame \\
Linear orders & \ittm & \dame & \ittm &\dame \\
Well-orders & \ittm & \dame & \dame & \dame \\
Boolean algebras & \ittm & \dame & \ittm &\dame \\
[0.5em]
\end{tabular}
\caption{\small Classes of structures $\mathcal{K}$ and the internal logic of the atomic sheaf topos ${\rm Sh}(\mathcal{K}_\lambda^{\rm op},J_{\rm at})$}\label{table-1}
\end{table}


By Theorem \ref{thm:main-stability-thm}, Shelah's Stability Spectrum Theorem {\cite[Theorem III.5.15]{She90}} becomes a choice spectrum.
Let $T$ be a countable complete first-order theory, and ${\rm Mod}(T)$ be the category of models of $T$ and elementary embeddings.
We have the following four-classification (Example \ref{exa:stable-spectra}):
\begin{enumerate}
\item If $T$ is $\om$-stable:
${\rm Sh}({\rm Mod}(T)_\lambda^{\rm op},J_{\rm at})$ satisfies the axiom of choice for every (external-)set indexed families.
\item If $T$ is superstable, but not $\om$-stable:
${\rm Sh}({\rm Mod}(T)_\lambda^{\rm op},J_{\rm at})$ satisfies the axiom of choice for $\lambda^+$-families if and only if $\lambda$ is greater than or equal to the least cardinal in which $T$ is stable.
\item If $T$ is stable, but not superstable:
${\rm Sh}({\rm Mod}(T)_\lambda^{\rm op},J_{\rm at})$ satisfies the axiom of choice for $\lambda^+$-families if and only if $\lambda^\om=\lambda$.
\item If $T$ is unstable:
for any infinite cardinal $\lambda$, ${\rm Sh}({\rm Mod}(T)_\lambda^{\rm op},J_{\rm at})$ does not satisfy the axiom of choice for $\lambda^+$-families.
\end{enumerate}

One of our key techniques is Mejak's choice-criterion \cite{Mej19} for atomic toposes.
Via the topological Galois representation \cite{Car16,CaLa19}, we further link this to Zapletal's choice-criterion \cite{Zap26} for Fraenkel--Mostowski permutation models in set theory:
For a well-behaved class $\mathcal{K}$ of structures (some mild conditions with the strong amalgamation property and a monster model $M_\K$) and infinite cardinals $\lambda$ and $\kappa$, the following five conditions are equivalent (Corollary \ref{cor:AC-criterion-main}).
\begin{enumerate}
\item The atomic topos ${\rm Sh}(\K_\lambda^{\rm op},J_{\rm at})$ satisfies the axiom of choice for (external) $\kappa$-families.
\item Every $\kappa$-fan $\{A\embed B_\alpha\}_{\alpha<\kappa}$ admits a cocone in $\K_\lambda$. 
\item For every $\kappa$-fan $\{A\embed B_\alpha\}_{\alpha<\kappa}$ in $\K_\lambda$, in the topos of continuous ${\rm Aut}(M_\K/A)$-actions, the product 
$\displaystyle{\prod_{\alpha<\kappa}{\rm Aut}(M_\K/A)/{\rm Aut}(M_\K/B_\alpha)}$ is nonempty.
\item Zapletal's AC-criterion: $(M_\K,\lambda)$ is dynamically $\kappa$-complete (Definition \ref{def:Zap-dynamical-completeness}).
\item The permutation model $W_\lambda[[M_\K]]$ satisfies the axiom of choice for (internal) $\kappa$-families.
\end{enumerate}

Similarly, the following five conditions are equivalent (Corollary \ref{cor:DC-criterion-main}).
\begin{enumerate}
\item The atomic topos ${\rm Sh}(\K_\lambda^{\rm op},J_{\rm at})$ satisfies the axiom of dependent choice.
\item Every $\om$-chain $\{A_n\embed A_{n+1}\}_{n\in\om}$ admits a cocone in $\K_\lambda$. 
\item For every $\om$-chain $\{A_n\embed A_{n+1}\}_{n\in\om}$ in $\K_\lambda$, in the topos of continuous ${\rm Aut}(M_\K/A_0)$-actions, the inverse limit, 
$\displaystyle{\lim_{n\in\om^{\rm op}}{\rm Aut}(M_\K/A_0)/{\rm Aut}(M_\K/A_n)}$, is nonempty.
\item Zapletal's DC-criterion: $(M_\K,\lambda)$ is DC-complete (Definition \ref{def:Zap-DC-completeness}).
\item The permutation model $W_\lambda[[M_\K]]$ satisfies the axiom of dependent choice.
\end{enumerate}

(1)$\Leftrightarrow$(2) is essentially due to Mejak \cite{Mej19} (see Theorems \ref{lem:cone-characterization-countable-choice} and \ref{lem:cone-characterization-dependent-choice}).
(2)$\Leftrightarrow$(3) is due to the topological Galois representation \cite{Car16,CaLa19} combined with Lemma \ref{lemma:cocone-in-slice-topos}.
For (3)$\Leftrightarrow$(4), we see that (4) is exactly the result of explicitly calculating the product or inverse limit in (3) (Propositions \ref{prop:AC-critetion-contG-Zap} and \ref{prop:DC-critetion-contG-Zap}).
(4)$\Leftrightarrow$(5) is due to Zapletal \cite{Zap26}.

Therefore, the characterization of the axiom of choice in terms of model-theoretic stability applies to Fraenkel--Mostowski permutation models as well.

\section{Atomic topos}

For readers approaching the subject from non-categorical logic, we give an introduction as elementary  as possible.
We explain the minimal tools needed in this paper; for further background on sheaves and topos theory, see \cite{Elephant,Elephant2,SGL}.
This section is particularly relevant to \cite[C3.5]{Elephant2}.

\subsection{Atomic topology}

A Grothendieck topology $J$ on a category $\C$ specifies a binary relation between a $\C$-object $p$ and a family $U$ of $\C$-morphisms into $p$, expressed as ``$U$ is a $J$-cover of $p$.''
In this article, the only Grothendieck topology we consider is the atomic topology.

\begin{definition}
Let $\C$ be a category.
A Grothendieck topology $J$ on $\C$ is called the {\bf atomic topology} if:
\begin{enumerate}
\item It is non-degenerate; that is, $\emptyset$ is not a $J$-cover of any object $p$.
\item For every morphism $q\edge{a}p$, the singleton $\{q\edge{a}p\}$ is a $J$-cover of $p$.
\end{enumerate}

We denote the atomic topology by $J_{\rm at}$.
We often call a morphism $q\edge{a}p$ a $J_{\rm at}$-cover of $p$.
Note that a family $U$ of morphisms into $p$ is a $J_{\rm at}$-cover of $p$ if and only if $U\not=\emptyset$.
\end{definition}

A pair $(\C,J)$ consisting of a small category $\C$ and a Grothendieck topology $J$ is called a {\bf site}.
When $J$ is the atomic topology, it is called an atomic site.

The categorical analogue of downward directedness is called the right Ore condition.
Recall that a cone over a family of morphisms $\{X_n\xedge{f_n} Y\}_{n\in I}$ is a family of morphisms $\{T\xedge{x_n} X_n\}_{n\in I}$ such that $f_n\circ x_n=f_m\circ x_m$ for all $n,m\in I$.
A category $\mathcal{C}$ satisfies the {\bf right Ore condition} if every pair of morphisms $\{ X_i\to Y\}_{i<2}$ admits a cone.
Equivalently, for any morphisms $a\edge{x}d$ and $b\edge{y}d$, there exist morphisms $c\edge{z}a$ and $c\edge{w}b$ such that $x\circ z=y\circ w$.
The dual notion is called the left Ore condition.

\begin{obs}\label{obs:right-Ore-dense-atomic}
The double-negation topology on a small category $\mathcal{C}$ satisfying the right Ore condition is the atomic topology.
\qed
\end{obs}

All of our concrete examples satisfy this condition, so every topology considered below is atomic.

\subsection{Sheaf}

A presheaf on a small category $\C$ is a functor $\C^{\rm op}\to{\bf Set}$.
We will be particularly interested in presheaves on $\K^{\rm op}$ for a class $\K$ of structures; these are functors $\K\to{\bf Set}$.

\begin{notation}
Let $X\colon\C^{\rm op}\to{\bf Set}$ be a presheaf.
For a $\C$-object $p$, we often write $X_p$ for $X(p)$.
For a $\C$-morphism $q\edge{a}p$, we often denote $X(p)\xedge{X(a)}X(q)$ by $-\at a$.
Thus it sends $x\in X_p$ to $x\at a\in X_q$.
\end{notation}

For presheaves $X,Y$ on $\C$, a morphism of presheaves $f\colon X\to Y$ is a natural transformation.
Explicitly, it is a family of functions $(f_p\colon X_p\to Y_p)_{p\in\C}$ such that $f_p(x)\at a=f_q(x\at a)$ for every $\C$-morphism $q\edge{a}p$ and every $x\in X_p$.

\begin{notation}
We denote the representable presheaf by $\yo p$.
That is, for each $p\in\C$, let $(\yo p)_q$ be the set ${\rm Hom}_\C(q,p)$ of all $\C$-morphisms from $q$ to $p$.
For $r\edge{a}q$ and $(q\edge{f}p)\in(\yo p)_q$, we have $f\at a=f\circ a$.

Yoneda's lemma says that
if $X$ is a presheaf on $\C$ then, for each $\C$-object $p$, the morphisms $\yo p\to X$ may be identified with $X_p$.
\end{notation}

Given a morphism $q\edge{a}p$, an {\bf $a$-matching} element is an element $x\in X_q$ such that, for all morphisms $s,t\colon r\to q$, if $a\circ s=a\circ t$, then $x\at s=x\at t$.

\begin{obs}\label{obs:characterization-atomic-sheaf-presheaf}
Let $\C$ be a small category satisfying the right Ore condition.
For a presheaf $X$ on a small category $\C$, the following two conditions are equivalent:
\begin{enumerate}
\item $X$ is a $J_{\rm at}$-sheaf.
\item For every $J_{\rm at}$-cover $q\edge{a}p$, every $a$-matching element $x\in X_q$ extends uniquely to an element $\hat{x}\in X_p$; that is, $\hat{x}\at a=x$.\qed
\end{enumerate}
\end{obs}

We shall take this as the definition of a $J_{\rm at}$-sheaf.
For a site $(\C,J)$, the category of $J$-sheaves on $\C$ is a topos, denoted by ${\rm Sh}(\C,J)$.

\begin{lemma}[Epimorphisms $=$ local surjections]\label{fact:SGL-epi-characterization-cover}
For a morphism $f\colon X\to Y$ in ${\rm Sh}(\mathcal{C},J_{\rm at})$, the following are equivalent:
\begin{enumerate}
\item $f$ is an epimorphism.
\item For every $\C$-object $p$ and every $y\in Y_p$, there exist a $J_{\rm at}$-cover $q\edge{a}p$ and an element $x\in X_q$ such that $f_q(x)=y\at a$.
\end{enumerate}
\end{lemma}

A Grothendieck topology $J$ on a small category $\C$ is {\bf subcanonical} if every representable presheaf $\yo X$ on $\C$ is a $J$-sheaf.
If $J_{\rm at}$ is subcanonical, we call $(\C,J_{\rm at})$ a subcanonical atomic site.

It is known that a topology is subcanonical if and only if all of its covering sieves are effective epimorphic (see e.g.~\cite[C2.2]{Elephant2}).
In the case of an atomic topology, it is characterized using the following notion.

\begin{definition}\label{def:prestrict}
A {\bf strict epimorphism} is a morphism $X\edge{f}Y$ such that the following holds:
Let $X\edge{g}Z$ be such that $f\circ x=f\circ y$ implies $g\circ x=g\circ y$ for any $x,y\colon T\to X$.
Then there is a unique $Y\edge{t}Z$ such that $g=t\circ f$.
\end{definition}

Note that a strict epimorphism is always an epimorphism, since if $g:=u\circ f=v\circ f$ then we get a unique $t$ such that $g=t\circ f$, which means $t=u=v$.
The dual of a strict epimorphism is called a {\bf strict monomorphism}.

\begin{lemma}[\cite{Mar25}]\label{lem:subcanonical-strict}
The following are equivalent:
\begin{enumerate}
\item The atomic topology $J_{\rm at}$ on $\C$ is subcanonical.
\item Every $\C$-morphism $A\edge{f} B$ is a strict epimorphism.
\end{enumerate}
\end{lemma}

\begin{proof}
Think of a morphism $X\edge{f}Y$ as a $J_{\rm at}$-cover of $Y$, and a morphism $X\edge{g}Z$ as an element $g\in(\yo Z)_X$.
Note that $g\at x$ is defined as $g \circ x$.
The premise of strictness is precisely that $g\in(\yo Z)_X$ is an $f$-matching element.
The conclusion of strictness is the unique existence of some $t\in(\yo Z)_Y$ such that $g=t\at f$.
In other words, Observation \ref{obs:characterization-atomic-sheaf-presheaf} can be restated as the assertion that all morphisms are strict.
\end{proof}

\subsection{Atom}

An atom is the categorical analogue of a singleton.

\begin{definition}
An object $X\not=\mathbf{0}$ is an {\bf atom} if the only subobjects of $X$ are $\mathbf{0}\mono X$ and $X\mono X$.
\end{definition}

\begin{obs}\label{obs:atomic-epi-nontrivial}~
\begin{enumerate}
\item If $X\not=\mathbf{0}$ and $A$ is an atom, then every morphism $X\to A$ is an epimorphism.
\item The image ${\rm Im}(f)\embed Y$ of a morphism $A\edge{f}Y$ from an atom $A$ is an atom.
\end{enumerate}
\end{obs}


A set $G$ of objects of a category $\C$ is a generating family if, for any $\C$-morphisms $f,g\colon X\to Y$, if $fx=gx$ holds for every morphism $(T\edge{x}X)\in G$, then we have $f=g$.

\begin{obs}\label{obs:yo-generator}
Let $(\C,J_{\rm at})$ be a subcanonical atomic site.
\begin{enumerate}
\item For every $\C$-object $p$, $\yo p$ is an atom.
\item $\{\yo p:p\in\C\}$ is a generating family for ${\rm Sh}(\C,J_{\rm at})$.
\item Every $J_{\rm at}$-sheaf $X$ is a coproduct of atoms.
\item For every atom $A$, there exists an epimorphism $\yo p\surj A$.
\end{enumerate}
\end{obs}

\begin{proof}
(1)
In the category of $J_{\rm at}$-sheaves, let $S\mono\yo p$ be a subobject distinct from $\mathbf{0}$.
Since $S\not=\mathbf{0}$, there exists some $(q\edge{a}p)\in S_q$.
Clearly, $q\edge{a}p$ is an $a$-matching element.
Since $S$ is a $J_{\rm at}$-sheaf, it extends uniquely to some $(p\edge{i}p)\in S_p$; i.e., $i\at a=i\circ a=a$.
By uniqueness, $i=1_p\colon p\to p$, and hence $S=\yo p$.

(2)
For morphisms $f,g\colon A\to B$, if $f\not=g$, there exist $p\in\C$ and $x\in A_p$ such that $f_p(x)\not=g_p(x)$.
By Yoneda's lemma, $x\in A_p$ may be regarded as a morphism $\tilde{x}\colon\yo p\to A$, and then $f\tilde{x}\not=g\tilde{x}$.

(3)
See e.g.~\cite[Proposition 2.3]{Car12}.

(4)
If $\{G_i\}_i$ is a generating family, there is a canonical epimorphism $\sum_{i,G_i\edge{a}A} G_i\surj A$.
Since $A\not=\mathbf{0}$, we must have $G_i\not=\mathbf{0}$ for some $i\in I$.
Then we get a morphism $G_i\to A$.
By Observation \ref{obs:atomic-epi-nontrivial} (1), this is an epimorphism.
Since $\{\yo p\}_p$ is a generating family, there is therefore an epimorphism $\yo p\surj A$ from an atom.
\end{proof}

Here, even without assuming subcanonicity, one can obtain the corresponding results by sheafifying $\yo p$.
Thus, for instance, Observation \ref{obs:yo-generator} (3) holds without the assumption of subcanonicity; see \cite[Proposition 2.3]{Car12}.

\section{Axioms of choice}

We introduce the axiom of choice within a topos.
For background and details, see also \cite[D4.5]{Elephant2} and \cite{Mej19}.

\begin{definition}[see e.g.~\cite{Mej19}]
Let $\mathcal{E}$ be a topos.
\begin{enumerate}
\item $\mathcal{E}$ satisfies the {\bf axiom of choice for $\kappa$-indexed families} for an (external) cardinal $\kappa$, if, for every set $I$ of cardinality $\kappa$ and every family of epimorphisms $\{g_i\colon X_i\surj Y_i\}_{i\in I}$, the product $\prod_{i\in I}g_i\colon\prod_{i\in I}X_i\to\prod_{i\in I}Y_i$ is also an epimorphism.
\item $\mathcal{E}$ satisfies the {\bf internal axiom of choice} if, whenever $g\colon X\surj Y$ is an epimorphism, $g^Z\colon X^Z\to Y^Z$ is also an epimorphism for every object $Z$.
\item $\mathcal{E}$ satisfies the {\bf external axiom of choice} if every epimorphism splits; that is, for every epimorphism $f\colon X\surj I$, there exists a morphism $I\edge{s}X$ such that $fs=1_I$.
\item $\mathcal{E}$ satisfies the {\bf internal axiom of dependent choice} if, for every sequence of epimorphisms $\cdots \surj X_2\surj X_1\surj X_0$, the projection from the limit $\lim_nX_n$ to $X_0$ is also an epimorphism.
\end{enumerate}
\end{definition}

Then, for $\kappa\geq\om$, we have the following (see e.g.~\cite{Mej19}):
\[
\xymatrix@R=10pt{
\mbox{\small External AC} \ar@{=>}[d]& \\ 
\mbox{\small Internal AC} \ar@{=>}[r]\ar@{=>}[d] & \mbox{\small AC for $\kappa$-families} \ar@{=>}[d]\\
\mbox{\small Internal DC} \ar@{=>}[r] & \mbox{\small AC for $\om$-families}
}
\]

Here, AC and DC stand for the axiom of choice and dependent choice, respectively.
The AC for $\om$-families is also known as the (internal) axiom of countable choice.
We now discuss sufficient conditions for the axioms of choice and dependent choice to hold in an atomic sheaf topos.

\begin{lemma}\label{lem:atomic-topos-embedding-to-epi}
Suppose that a category $\C$ satisfies the right Ore condition and that the atomic topology $J_{\rm at}$ is subcanonical.
For every $\C$-morphism $X\edge{f} Y$, the morphism $\yo X\xedge{\yo f}\yo Y$ is an epimorphism in ${\rm Sh}(\C,J_{\rm at})$.
\end{lemma}

\begin{proof}
For each $\C$-object $A$, $(\yo f)_A\colon (\yo X)_A\to(\yo Y)_A$ is given by $a\mapsto f\circ a$.
To prove that $\yo f$ is locally surjective (Lemma \ref{fact:SGL-epi-characterization-cover}), let $y\in(\yo Y)_A$ be given, which is a morphism $A\edge{y}Y$.
Applying the right Ore condition to $X\edge{f}Y$ and $A\edge{y}Y$, we obtain morphisms $B\edge{a}A$ and $B\edge{x}X$ such that $f\circ x=y\circ a$.
Note that $B\edge{a}A$ is a $J_{\rm at}$-cover of $A$ and that $x\in(\yo X)_B$.
We then have $(\yo f)_B(x)=f\circ x=y\circ a=y\at a$.
This proves that $\yo f$ is locally surjective, hence an epimorphism.
\end{proof}

\begin{theorem}[\cite{Mej19}]\label{lem:cone-characterization-countable-choice}
Suppose that the atomic topology $J_{\rm at}$ on a category $\mathcal{C}$ is subcanonical.
Then, for a set $I$, the following are equivalent:
\begin{enumerate}
\item In $\C$, every family of morphisms $\{p_n\edge{a_n}p\}_{n\in I}$ with a common codomain admits a cone.
\item ${\rm Sh}(\mathcal{C},J_{\rm at})$ satisfies the axiom of choice for $I$-indexed families.
\end{enumerate}
\end{theorem}

\begin{proof}
(1)$\Rightarrow$(2)
To prove the axiom of choice for $I$-indexed families in the topos ${\rm Sh}(\mathcal{C},J_{\rm at})$, we show that an $I$-indexed product of epimorphisms is an epimorphism.
By Lemma \ref{fact:SGL-epi-characterization-cover}, given a family of local surjections $\{f_n\colon X_n\surj Y_n\}_{n\in I}$, it suffices to prove that $f:=\prod_nf_n\colon\prod_nX_n\to\prod_nY_n$ is locally surjective.
Let $p$ be a $\C$-object, and $y=(y_n)_{n\in I}\in\prod_n Y_{n,p}$ be given.

By local surjectivity of $f_n\colon X_n\to Y_n$ for each $n$, there exist a $J_{\rm at}$-cover $p_n\edge{a_n}p$ of $p$ and $x_n\in X_{n,p_n}$ such that $f_{n,p_n}(x_n)=y_n\at{a_n}$.
By assumption, the family $\{a_n\}_{n\in I}$ of morphisms admits a cone $\{q\edge{t_n}p_n\}_{n\in I}$.
Let $a=(q\edge{t_0}p_0\edge{a_0}p)$.
By the definition of a cone, $a=a_n\circ t_n$ holds for every $n$.
Now set $x_n'=x_n\at{t_n}\in X_{n,q}$.
Then we get
\[f_{n,q}(x_n')=f_{n,q}(x_n\at{t_n})=f_{n,p_n}(x_n)\at{t_n}=y_n\at{a_n}\at{t_n}=y_n\at a.\]

Thus, if we define $x=(x_n')_{n\in I}$, then $x\in\prod_nX_{n,q}$ and $f_q(x)=(y_n\at a)_{n\in I}=y\at a$.
This proves that $f$ is locally surjective.

(2)$\Rightarrow$(1):
Let $\{p_n\edge{a_n}p\}_{n\in I}$ be a family of $\C$-morphisms.
By Lemma \ref{lem:atomic-topos-embedding-to-epi}, $\yo a_n\colon \yo p_n\to \yo p$ is an epimorphism.
The axiom of choice implies that the $I$-indexed product $\Psi=\prod_{n\in I}\yo(a_n)\colon\prod_{n\in I}\yo p_n\to\prod_{n\in I}\yo p$ is also an epimorphism.

Since $1_p\in(\yo p)_p$, we have $(1_p)_{i\in I}\in(\prod_{n\in I}\yo p)_p$.
By local surjectivity of $\Psi$, there exist a $J_{\rm at}$-cover $q\edge{c}p$ and elements $x_n\in(\yo p_n)_q$ such that $\Psi_q((x_n)_{n\in I})=1_p\at c$.
The $n$th component of the left-hand side is $(\yo a_n)_q(x_n)=a_n\circ x_n$, while the right-hand side is $1_p\at c=1_p\circ c=c$.
Consequently, $a_n\circ x_n=c$ for every $n\in I$, and hence $\{q\edge{x_n}p_n\}_{n\in I}$ is a cone.
\end{proof}

Recall that a cone over a sequence of morphisms $\{X_{n+1}\edge{f_n} X_n\}_{n\in\N}$ is a family of morphisms $\{T\edge{x_n} X_n\}_{n\in\N}$ such that $f_n\circ x_{n+1}=x_n$ for every $n\in\N$.

\begin{theorem}[\cite{Mej19}]\label{lem:cone-characterization-dependent-choice}
Suppose that the atomic topology $J_{\rm at}$ on a category $\mathcal{C}$ is subcanonical.
Then the following are equivalent.
\begin{enumerate}
\item In $\mathcal{C}$, every sequence of morphisms $\{p_{n+1}\edge{a_n} p_n\}_{n\in\N}$ admits a cone.
\item ${\rm Sh}(\mathcal{C},J_{\rm at})$ satisfies the internal axiom of dependent choice.
\end{enumerate}
\end{theorem}

\begin{proof}
(1)$\Rightarrow$(2):
To prove the axiom of dependent choice in the topos ${\rm Sh}(\mathcal{C},J_{\rm at})$, let $\{X_\infty\edge{\pi_n}X_n\}_{n\in\N}$ be the limit of a sequence of epimorphisms $\{f_n\colon X_{n+1}\surj X_n\}_{n\in\N}$.
We shall show that the projection $\pi_0\colon X_\infty\to X_0$ is an epimorphism.
Limits of sheaves are computed componentwise; thus $X_{\infty,p}$ is the limit of $\{X_{n+1,p}\xedge{f_{n,p}}X_{n,p}\}$ in ${\bf Set}$.

To verify that the projection $\pi_0$ is locally surjective, take a $\C$-object $p_0$ and an element $y_0\in X_{0,p_0}$.
Inductively assume that we have chosen an object $p_n$ and an element $y_n\in X_{n,p_n}$.
By local surjectivity of $f_n$, there are a $J_{\rm at}$-cover $p_{n+1}\edge{a_n}p_n$ of $p_n$ and an element $y_{n+1}\in X_{n+1,p_{n+1}}$ such that $f_{n,p_{n+1}}(y_{n+1})=y_n\at{a_n}$.

By assumption, the sequence $\{p_{n+1}\edge{a_n}p_n\}_{n\in\N}$ of morphisms admits a cone $\{p\edge{t_n}p_n\}_{n\in\N}$.
In particular, $a_n\circ t_{n+1}=t_n$.
Then
\begin{align*}
f_{n,p}(y_{n+1}\at{t_{n+1}})&=f_{n,p_{n+1}}(y_{n+1})\at{t_{n+1}}\\
&=y_n\at{a_n}\at{t_{n+1}}=y_n\at{a_n\circ t_{n+1}}=y_n\at{t_n}.
\end{align*}
This means $x=(y_n\at{t_n})_{n\in\N}\in X_{\infty,p}$.
By the definition of $x$, $\pi_{0,p}(x)=y_0\at{t_0}$, proving that $\pi_0$ is locally surjective.

(2)$\Rightarrow$(1):
Let $(p_{n+1}\edge{a_n}p_n)_{n\in\N}$ be a sequence of $\C$-morphisms.
By Lemma \ref{lem:atomic-topos-embedding-to-epi}, $\yo a_n\colon \yo p_{n+1}\to\yo p_n$ is an epimorphism.
By the axiom of dependent choice, $\pi_0\colon\lim_n\yo p_n\to\yo p_0$ is also an epimorphism.

Now take $1_{p_0}\in(\yo p_0)_{p_0}$.
By local surjectivity of $\pi_0$, there exist a $J_{\rm at}$-cover $q\edge{c}p_0$ and an element $x\in(\lim_n\yo p_n)_q$ such that $\pi_0(x)=1_{p_0}\at c=1_{p_0}\circ c=c$.
By the definition of the limit, the element is of the form $x=(q\edge{x_n} p_n)_{n\in\N}$, where $x_n=a_n\circ x_{n+1}$ and $x_0=c$.
It follows that $\{q\edge{x_n}p_n\}_{n\in\om}$ is a cone.
\end{proof}

\begin{remark}
The implications (1)$\Rightarrow$(2) in Theorem \ref{lem:cone-characterization-countable-choice} and Theorem \ref{lem:cone-characterization-dependent-choice} are due to Mejak \cite{Mej19}.
Subcanonicity is used only in the proof of (2)$\Rightarrow$(1).
\end{remark}

%
%
%
%

\section{Structures and embeddings}

\subsection{Classes of structures as categories}\label{sec:class-str-cat}

In this article, we are primarily interested in the category of structures in $\K$ and embeddings, where $\mathcal{K}$ is a fixed class of structures.
If the reader wishes to avoid abstract arguments, the reader may simply think of $\K$ as a {\bf class of first-order structures}, and skip Section \ref{sec:class-str-cat}.

We declare the following as the minimum requirements that must be satisfied by what is called a structure:
\begin{itemize}
\item Every structure $A$ is a set $|A|\in{\bf Set}$ equipped with additional data.
\item Transportability \cite[Definition 5.28]{Joy}: If $A$ is a structure, $X$ is a set, and $\sigma\colon |A|\to X$ is a renaming of elements (i.e., a bijection), then the additional data on $A$ can be transported to $X$; that is, there is a structure $A^\sigma$ such that $|A^\sigma|=X$ and $\sigma$ lifts to an isomorphism $A\edge{\simeq} A^\sigma$.
\end{itemize}

We also declare the following as the minimum requirements that must be satisfied by what is called an embedding:
\begin{itemize}
\item Every homomorphism $h\colon A\to B$ is a function $|h|\colon |A|\to |B|$ satisfying a certain condition.
\item Every embedding is an injective homomorphism satisfying a certain condition.
\item Every surjective embedding must be an isomorphism.
\end{itemize}

Therefore, we proceed with the discussion within the following framework.


\begin{definition}
A {\bf concrete category of structures and embeddings} (CCSE) is a category $\K$ equipped with a faithful functor $|-|\colon\K\to{\bf Set}$ such that:
\begin{enumerate}
\item If $f\colon A\to B$ is a morphism in $\K$, then $|f|\colon |A|\to |B|$ is an injection.
\item $|-|$ is conservative; that is, if $|f|\colon|A|\to|B|$ is a bijection, then $f\colon A\to B$ is an isomorphism in $\K$.
\item $|-|$ is transportable (a.k.a.~an isofibration); that is, for every $A\in\K$ and every bijection $\sigma\colon |A|\to X$ in ${\bf Set}$, there is an isomorphism $A\edge{\bar{\sigma}}A^\sigma\in\K$ such that $|A^\sigma|=X$ and $|\bar{\sigma}|=\sigma$.
\end{enumerate}

We often write simply $A$ for $|A|$ and $f$ for $|f|$.
An object in $\K$ is called a {\bf structure}, and a morphism $A\to B$ in $\K$ is called an {\bf embedding} or an {\bf extension}, often written as $A\embed B$.
\end{definition}

For an infinite cardinal $\lambda$, let $\K_\lambda$ be the class of all structures in $\K$ of cardinality at most $\lambda$.
\[
\K_\lambda=\{A\in\K:\mbox{the cardinality of the carrier $|A|$ is at most $\lambda$}\}.
\]
We often think of $\K_\lambda$ as a full subcategory of $\K$.

Since $|-|\colon\K\to{\bf Set}$ is faithful, $\K$ is locally small, but size issues may still arise.
For this reason, we usually assume the following condition.

\begin{definition}
A {\bf structure-class} is a CCSE $(\K,|-|)$ which is fiber-small; that is, $\{A\in\K:|A|=X\}$ is a set for any $X$.
\end{definition}

\begin{obs}\label{obs:essentially-small}
If $\K$ is a structure-class, then $\K_\lambda$ is essentially small for any cardinal $\lambda$.
\end{obs}

\begin{proof}
Given $A\in\K_\lambda$, there is a bijection $\sigma\colon |A|\edge{\simeq}\alpha$ for some $\alpha\leq\lambda$.
By transportability, there is an isomorphism $A\edge{\simeq} A^\sigma\in\K_\lambda$ such that $|A^\sigma|=\alpha$.
Let $\K_\lambda^-$ be the full subcategory of $\K_\lambda$ consisting of structures whose carrier are ordinals $\alpha\leq\lambda$.
Then $\K_\lambda^-$ is small, since $\K_\lambda^-$ is a set by fiber-smallness and $|-|$ is faithful.
It is easy to check that $\K_\lambda^-$ is equivalent to $\K_\lambda$.
\end{proof}

Except for Section \ref{sec:set-theory}, transportability and fiber-smallness are used only in the proof of Observation \ref{obs:essentially-small}, so all arguments remain valid even if these axioms are replaced by the essential smallness of $\K_\lambda$ for all cardinals $\lambda$.
When considering the category of (pre-)sheaves on $\K_\lambda^{\rm op}$, we identify $\K_\lambda$ with the equivalent small category $\K_\lambda^-$.


\begin{example}
For every first-order theory $T$, the category of models of $T$ and (elementary) embeddings is a structure-class.
\end{example}

Therefore, we consider a structure-class as an abstraction of a class of structure.
Here, note that transportability corresponds to the isomorphism invariance which always holds for abstract elementary classes (AECs) in model theory \cite{She09}.
We always assume that $\K$ is a structure-class.

\begin{remark}
The notion of structure-classes has some overlap with the framework of ``concrete accessible categories with concrete monomorphisms'' in the sense of \cite{LiRo16}.
There are some differences: Our definition does not assume accessibility and uses the cardinality of the underlying set as the size measure instead of a presentability rank (thus, the downward L\"owenheim-Skolem property is not assumed), and our definition assumes transportability, which has a significant role in Section \ref{sec:set-theory}.
\end{remark}

\subsection{Amalgamation}
In model theory, the left Ore condition is known as the amalgamation property.

\begin{definition}
A structure-class $\mathcal{K}$ has the {\bf amalgamation property} ({\rm AP}) if, for all structures $A,B,C\in\mathcal{K}$ and all embeddings $A\embedge{f}B$ and $A\embedge{g}C$, there exist a structure $D\in\mathcal{K}$ and embeddings $B\embedge{h}D$ and $C\embedge{k}D$ such that $h\circ f=k\circ g$.
\[
\xymatrix{
	A \ar[rr]^f \ar[d]_g & & B \ar[d]^h \\
	C \ar[rr]_k & & D
}
\]
\end{definition}

Such $(h,k)$ is often called an amalgamation of $(f,g)$.

\begin{obs}\label{obs:right-Ore-dense-atomic2}
$\mathcal{K}$ has the {AP} if and only if it satisfies the left Ore condition.

In this case, the opposite category $\mathcal{K}^{\rm op}$ satisfies the right Ore condition, so by Observation \ref{obs:right-Ore-dense-atomic}, the double-negation topology is the atomic topology.
Indeed, each embedding $A\embed B$ gives a cover of $A$.
\qed
\end{obs}


A class of structures often has the following stronger property, which has also been extensively studied in model theory:

\begin{definition}
A structure-class $\mathcal{K}$ has the {\bf strong amalgamation property} (strong AP) if every pair $(A\embedge{f}B,\ A\embedge{g}C)$ has an amalgamation $(B\embedge{h}D,\ C\embedge{k}D)$ such that the intersection of the images of $B$ and $C$ in $D$ coincides with the image of $A$ in $D$ in ${\bf Set}$:
\[h[B]\cap k[C]=hf[A]=kg[A].\]
\end{definition}


The following condition is always satisfied by the class $\K$ of all models of a first-order theory, so it usually receives little attention; however, in our context, it must be explicitly stated.

\begin{definition}[\cite{LiRo16}]
A structure-class $\K$ satisfies the {\bf coherence axiom} if for every $\K$-morphisms $A\embedge{g}C$ and $B\embedge{h}C$, and every function $f\colon |A|\to|B|$, if $|g|=|h|\circ f$ then there exists a $\K$-morphism $\bar{f}\colon A\to B$ such that $|\bar{f}|=f$.
\end{definition}

Indeed, the coherence axiom is a condition explicitly contained in the definition of AEC in model theory \cite{She09}.
Therefore, every AEC forms a structure-class satisfying the coherence axiom.

\begin{prop}\label{prop:stAP-subcanonical}
Let $\mathcal{K}$ be a structure-class such that $\K_\lambda$ satisfies the strong AP and the coherence axiom.
Then the atomic topology $J_{\rm at}$ on $\mathcal{K}_\lambda^{\rm op}$ is subcanonical.
\end{prop}

\begin{proof}
For a structure $A\in\mathcal{K}$, we show that $\yo A$ is a sheaf.
Since we are working in the opposite category $\mathcal{K}^{\rm op}$, $(\yo A)_B$ is the set of all embeddings from $A$ to $B$.
We use Observation \ref{obs:characterization-atomic-sheaf-presheaf} to show that $\yo A$ is a $J_{\rm at}$-sheaf.

Take an embedding $P\embedge{i}Q$ and an $i$-matching element $x\in(\yo A)_Q$.
By the definition of an $i$-matching element in the opposite category, $si=ti$ implies $sx=tx$.
By assumption, we obtain a strong amalgamation $h,k$ of two copies of $P\embedge{i}Q$.
\[
\xymatrix{
	P \ar[rr]^i \ar[d]_i & & Q \ar[d]^h \\
	Q \ar[rr]_k & & R
}
\]

Here, $h[Q]\cap k[Q]=hi[P]$ in ${\bf Set}$.
For every $a\in A$, we have $x(a)\in Q$.
By the definition of an $i$-matching element, the amalgamation identity $hi=ki$ implies $hx(a)=kx(a)\in h[Q]\cap k[Q]=hi[P]$.
Since $h$ is injective, $x(a)\in i[P]$.
Thus $x[A]\subseteq i[P]$.

Since $i$ is injective, we have a function $i^-\colon i[P]\to P$ such that $i^-(i(z))=z$ in ${\bf Set}$.
Then the function $x':=i^{-}\circ x\colon A\to P$ satisfies $|i|\circ x'=|x|$.
By coherence, there is a morphism $\hat{x}\colon A\to P$ such that $|\hat{x}|=x'$.
Then $i\circ \hat{x}=x$, and $\hat{x}\in(\yo A)_P$.
Here, such $\hat{x}$ is unique since $i$ is an monomorphism.
The condition in Observation \ref{obs:characterization-atomic-sheaf-presheaf} is therefore satisfied.
\end{proof}

We say that a structure-class $\K_\lambda$ is a {\bf subcanonical AP class} if it satisfies the coherence axiom and has the AP, and the atomic topology $J_{\rm at}$ on $\K_\lambda^{\rm op}$ is subcanonical.

\begin{definition}
$\mathcal{K}$ has the {\bf joint embedding property} (JEP) if, for any structures $A,B\in\mathcal{K}$, there exists a structure $C\in\mathcal{K}$ such that $A,B\embed C$.
\end{definition}

\begin{prop}\label{prop:JEP-inhabited}
If $\mathcal{K}_\lambda$ is a subcanonical AP class, the following two conditions are equivalent:
\begin{enumerate}
\item $\mathcal{K}_\lambda$ has the JEP.
\item In $\mathrm{Sh}(\mathcal{K}_\lambda^{\rm op},J_{\rm at})$, every representable presheaf $\yo X$ is inhabited.
\end{enumerate}
\end{prop}

\begin{proof}
For $\yo X$ to be inhabited, that is, for $\yo X\to\mathbf{1}$ to be an epimorphism (equivalently, locally surjective by Lemma \ref{fact:SGL-epi-characterization-cover}), means that for every structure $P\in\mathcal{K}_\lambda$ there is a cover $P\embed Q$ such that $(\yo X)_Q\not=\emptyset$.
Equivalently, there exists an embedding $X\embed Q$.
In other words, for every structure $P$, there is a structure $Q$ together with embeddings $P\embed Q$ and $X\embed Q$.
Thus every representable presheaf is inhabited (that is, $\yo X\to\mathbf{1}$ is an epimorphism for every $X\in\mathcal{K}_\lambda$) if and only if every pair $X,P\in\mathcal{K}_\lambda$ embeds into some $Q\in\mathcal{K}_\lambda$, which is precisely the assertion that $\mathcal{K}_\lambda$ has the joint embedding property.
\end{proof}

\begin{remark}
If we view a cocone as a generalized amalgamation, then Mejak's choice criterion (Theorems \ref{lem:cone-characterization-countable-choice} and \ref{lem:cone-characterization-dependent-choice}) in $\K^{\rm op}$ can also be regarded as a generalized amalgamation property for $\K$.
\end{remark}

\section{Choice versus stability}\label{sec:choice-stable}

\subsection{Type and stability}

In model theory, a type is a consistent list of properties satisfied by an element (in an extension) of a structure.
More precisely, given a structure $M$ and a subset $A\subseteq M$, it is the set of all formulas $\varphi(z,\bar{a})$ with parameters from $A$ that are satisfied by an element $z\in M$.
There are several equivalent definitions of a type; we adopt the following one, which is particularly well suited to the atomic topology.

\begin{definition}
Let $\mathcal{K}$ be a structure-class.
\begin{enumerate}
\item For structures $A,M\in\mathcal{K}$, a pair $(f,z)$ consisting of an embedding $A\embedge{f}M$ and an element $z\in M$ is called a {\bf pointed extension} (over $A$).
\item Pointed extensions $(A\embedge{f}M,s)$ and $(A\embedge{g}N,t)$ are equivalent if there exists an amalgamation $M\embedge{h}L$ and $N\embedge{k}L$ of $f$ and $g$ such that $h(s)=k(t)$.
\end{enumerate}
\end{definition}

\begin{lemma}
If $\K$ has the AP, the equivalence between pointed extensions over $A$ yields an equivalence relation.
\end{lemma}

\begin{proof}
To show transitivity, let $(A\embedge{m_i}M_i;a_i)_{i\leq 2}$ be pointed extensions such that $(m_0;a_0)\simeq(m_1;a_1)\simeq(m_2;a_2)$.
Then $(m_0,m_1)$ has an amalgamation $(u_0,u_1)$, and $(m_1,m_2)$ has an amalgamation $(v_1,v_2)$ such that $u_0(a_0)=u_1(a_1)$ and $v_1(a_1)=v_2(a_2)$.
By the AP, $(u_1,v_1)$ has an amalgamation $(s,t)$:
\begin{align*}
\xymatrix{
	M_0 \ar[d]_{u_0} & & A \ar[rr]^{m_2} \ar[ll]_{m_0} \ar[d]_{m_1} & & M_2 \ar[d]^{v_2}\\
	N_{01} \ar[rrd]_{s} & & M_1 \ar[rr]_{v_1} \ar[ll]^{u_1} & & N_{12} \ar[lld]^{t}\\
	& & L & &
}
\end{align*}

Then, $(su_0,tv_2)$ is an amalgamation of $(m_0,m_2)$, and we have  $su_0(a_0)=su_1(a_1)=tv_1(a_1)=tv_2(a_2)$.
Therefore, we get $(m_0;a_0)\simeq(m_2;a_2)$.
\end{proof}

\begin{definition}
An equivalence class of pointed extensions $(A\embedge{f}M,z)$ with domain $A$ is called a {\bf Galois type} over $A$.
\end{definition}

\begin{remark}
Galois types for elementary embeddings coincide with the usual notion of types in model theory \cite[Chapter 8]{Bal09}.
Galois types for embeddings (for a suitable class of structures) are known to correspond to {\bf quantifier-free types} (see \cite[Remark 3.8]{Vas17}).
\end{remark}

\begin{remark}
We use the $\lambda$-restricted version of Galois type.
For an AEC with the usual L\"owenheim--Skolem hypotheses this agrees with the standard definition of Galois type above the L\"owenheim--Skolem cardinal.
\end{remark}

\begin{definition}
$\K$ is {\bf Galois $\lambda$-stable} if, for every structure $A\in\K_\lambda$, there are at most $\lambda$ Galois types over $A$ (in $\K_\lambda$).
\end{definition}

\begin{example}\label{exa:LO-non-stable}
The class $\mathcal{L}$ of linear orders is not Galois $\om$-stable:
consider the ordering of the rationals $\mathbb{Q}\in\mathcal{L}_\om$.
For each real number $r\in\mathbb{R}$, the suborder $\mathbb{Q}\cup\{r\}\subseteq\mathbb{R}$ is also a countable linear order.
Clearly, for each $r\in\R$, $(\mathbb{Q}\embed\mathbb{Q}\cup\{r\};r)$ gives a distinct Galois type over $\mathbb{Q}$.
Thus there are uncountably many distinct Galois types over $\mathbb{Q}$, so $\mathcal{L}$ is not $\om$-stable.
One may also show that it is not Galois $\lambda$-stable for any $\lambda$; see e.g.~\cite{Kei76}.
\end{example}


\begin{theorem}\label{thm:main-stability-thm}
For an infinite cardinal $\lambda$, suppose that $\K_\lambda$ is a subcanonical AP class satisfying the $\lambda$-TV (Definition \ref{def:TV}). Then the following three conditions are equivalent.
\begin{enumerate}
\item $\K$ is Galois $\lambda$-stable.
\item ${\rm Sh}(\K^{\rm op}_\lambda,J_{\rm at})$ satisfies the axiom of choice for $\lambda^+$-indexed families.
\item ${\rm Sh}(\K^{\rm op}_\lambda,J_{\rm at})$ satisfies the axiom of choice for all (external-)set indexed families.
\end{enumerate}
Here $\lambda^+$ denotes the least cardinal greater than $\lambda$.
\end{theorem}


(3)$\Rightarrow$(2) is immediate.
We first prove (2)$\Rightarrow$(1).

\begin{theorem}\label{thm:choice-implies-stability}
Suppose that $\K_\lambda$ is a subcanonical AP class.
If ${\rm Sh}(\K^{\rm op}_\lambda,J_{\rm at})$ satisfies the axiom of choice for $\lambda^+$-indexed families, then $\K$ is Galois $\lambda$-stable.
\end{theorem}

\begin{proof}
Suppose that $\K$ is not Galois $\lambda$-stable.
Then there are at least $\lambda^+$ Galois types over some $A\in\K_\lambda$.
Thus, for each $\alpha<\lambda^+$, we may choose a pointed extension $(A\embedge{e_\alpha}M_\alpha;z_\alpha)$, with all the resulting Galois types distinct.

If ${\rm Sh}(\K_\lambda^{\rm op},J_{\rm at})$ satisfies the axiom of choice for $\lambda^+$-indexed families, by subcanonicity, Theorem \ref{lem:cone-characterization-countable-choice} implies that $(A\embedge{e_\alpha}M_\alpha)_{\alpha<\lambda^+}$ admits a cocone $(M_\alpha\embedge{d_\alpha}N)_{\alpha<\lambda^+}$ in $\K_\lambda$.
Thus $d_\alpha e_\alpha=d_\beta e_\beta$ for all $\alpha,\beta<\lambda^+$.
If $\alpha\not=\beta$, the Galois types of $(e_\alpha;z_\alpha)$ and $(e_\beta;z_\beta)$ are distinct, so we must have $d_\alpha(z_\alpha)\not=d_\beta(z_\beta)$.
But $\{d_\alpha(z_\alpha):\alpha<\lambda^+\}\subseteq N\in\K_\lambda$, which would imply $\lambda^+\leq\lambda$, a contradiction.

It follows that ${\rm Sh}(\K_\lambda^{\rm op},J_{\rm at})$ does not satisfy the axiom of choice for $\lambda^+$-indexed families.
\end{proof}

We next give an equivalent condition for Theorem \ref{thm:main-stability-thm} (3).

\begin{definition}\label{def:universal-extension-prop}
A {\bf universal extension} of $M\in\K_\lambda$ is an extension $M\embedge{u} U_M\in\K_\lambda$ satisfying the following condition:
for every embedding $M\embedge{f} N\in\K_\lambda$, there exists an embedding $N\embedge{g} U_M$ such that $g\circ f=u$.

The class $\K_\lambda$ has the {\bf universal extension property} if every structure $M\in\K_\lambda$ has a universal extension.
\end{definition}

\begin{example}\label{exa:set-universal-extension}
The class $\mathrm{Set}$ of sets has the universal extension property:
for a set $M\in\mathrm{Set}_\lambda$, consider the inclusion map $M\embed M+\lambda=:U_M$.
For every injection $M\monoedge{f}N\in\mathrm{Set}_\lambda$, the cardinality of $N\setminus f[M]$ is at most $\lambda$, so there is necessarily an injection $N\setminus f[M]\mono \lambda\embed U_M$.
\end{example}

\begin{theorem}\label{thm:choice-vs-universal-extension}
Suppose that $\K_\lambda$ is a subcanonical AP class.
Then the following are equivalent:
\begin{enumerate}
\item $\K_\lambda$ has the universal extension property.
\item ${\rm Sh}(\K_\lambda^{\rm op},J_{\rm at})$ satisfies the axiom of choice for any (external-)set indexed families.
\end{enumerate}
\end{theorem}

\begin{proof}
($\Rightarrow$)
By Theorem \ref{lem:cone-characterization-countable-choice}, it suffices to show that every family of morphisms $(M\embedge{f_i} N_i)_{i\in I}$ admits a cocone in $\K_\lambda$.
Choose $M\embedge{u}U_M$ satisfying the universal extension property.
For each $i\in I$, there is then an embedding $N_i\embedge{g_i}U_M$ such that $g_i\circ f_i=u$.
Thus $(N_i\embedge{g_i}U_M)_{i\in I}$ is the desired cocone.

($\Leftarrow$)
Let $\K_\lambda^-$ be a small full subcategory of $\K_\lambda$ as in Observation \ref{obs:essentially-small}.
For $M\in\K_\lambda^-$, take the family $(M\embedge{i_k} N_k)_{k\in K}$ of all extensions in $\K_\lambda^-$.
This is possible because $\K_\lambda^-$ is small.
By Theorem \ref{lem:cone-characterization-countable-choice} with subcanonicity, the assumption (2) implies the existence of a cocone $(N_k\embedge{g_k} U_M)_{k\in K}$.
Set $u:=g_k\circ i_k\colon M\embed U_M$.
Since this is a cocone, the definition of $u$ is independent of $k$.

Every extension $M\embed N\in\K_\lambda^-$ is of the form $M\embedge{i_k}N_k$, and for $N_k\embedge{g_k}U_M$ we have $g_k\circ i_k=u$.
Therefore, $\K_\lambda^-$ has the universal extension property.
\end{proof}

\begin{remark}
By Theorem \ref{thm:choice-implies-stability} and Theorem \ref{thm:choice-vs-universal-extension}, the universal extension property implies Galois $\lambda$-stability.
\end{remark}

\subsection{Universal extension theorem}

In the setting of classical model theory, it is known that $\lambda$-stability implies the universal extension property.

\begin{fact}[{Shelah \cite[Claim II.1.16]{She09}}; see also {\cite[Chapter 10]{Bal09}}]\label{fact:Shelah-AEC}
Suppose that $\mathcal{K}$ is an abstract elementary class (AEC), $\K_\lambda\not=\emptyset$, and it has the AP.
Then Galois $\lambda$-stability implies the universal extension property.
\end{fact}

However, since this is a theorem in the context of AECs, it needs to be adjusted to be consistent with our setting.
The following is an adaptation of one of the AEC axioms to our setting.

\begin{definition}\label{def:TV}
A structure-class $\K$ satisfies the {\bf Tarski-Vaught $\lambda$-chain axiom} ($\lambda$-TV) if for every ordinal $0<\delta\leq\lambda$ and every chain $(X_i\embed X_j)_{i<j<\delta}$ in $\K_\lambda$:
\begin{enumerate}
\item $\colim_{i<\delta}X_i$ exists in $\K$.
\item The canonical map $|\colim_{i<\delta}X_i|\to\colim_{i<\delta}|X_i|$ is bijective.
\end{enumerate}
\end{definition}

In other words, $\K$ has concrete $\lambda$-chain colimits.
Replacing $\K$ in condition (1) with $\K_\lambda$ makes no difference:
This is because the union of sequences of length at most $\lambda$ formed from sets of cardinality at most $\lambda$ has cardinality at most $\lambda$.

\begin{theorem}\label{thm:Shelah-AEC}
Let $\K$ be a structure-class.
Suppose that $\K_\lambda$ satisfies the AP and the $\lambda$-TV.
If $\K$ is Galois $\lambda$-stable, then $\K_\lambda$ has the universal extension property.
\end{theorem}

We prepare for the proof of this theorem.

\begin{definition}
Let $p$ be a Galois type over $M\in\K_\lambda$.
An element $a\in|N|$ in an extension $M\embedge{i} N\in\K_\lambda$ {\bf realizes} $p$ if $p$ is the Galois type of $(i;a)$.
\end{definition}

\begin{definition}
An extension $A\embedge{i}S$ in $\K_\lambda$ is a {\bf pre-saturation} if, for every Galois type $p$ over $A$, some element $a\in |S|$ in the extension $i$ realizes $p$.
\end{definition}

The following can be viewed as an abstraction of \cite[Lemma 10.3]{Bal09}.

\begin{lemma}\label{lem:pre-saturation}
Suppose that $\K_\lambda$ satisfies the AP and the $\lambda$-TV.
If $\K$ is Galois $\lambda$-stable, then every structure $A\in\K_\lambda$ has a pre-saturation $A\embed S\in\K_\lambda$.
\end{lemma}

\begin{proof}
By Galois $\lambda$-stability of $\K$, there is a $\lambda$-indexed enumeration $\{p_\alpha\}_{\alpha<\lambda}$ of all Galois types over $A\in\K_\lambda$.
For each $\alpha<\lambda$ choose a representative $(A\embedge{i_\alpha} M_\alpha;c_\alpha)$ from $p_\alpha$.
We construct a chain of structures $(S_\alpha)_{\alpha\leq\lambda}$ and embeddings $A\embedge{e_\alpha} S_\alpha$ in $\K_\lambda$.

Let $S_0=A$ and $e_0=1_A$.
At each successor step, 
$A\embedge{e_\alpha}S_\alpha$ and $A\embedge{i_\alpha}M_\alpha$ are amalgamated by $S_\alpha\embedge{s_\alpha}S_{\alpha+1}$ and $M_\alpha\embedge{v_\alpha}S_{\alpha+1}$.
Then let $e_{\alpha+1}$ be the composition $A\embedge{e_\alpha}S_\alpha\embedge{s_\alpha}S_{\alpha+1}$.
Then the element $v_\alpha(c_\alpha)\in S_{\alpha+1}$ in the extension $e_{\alpha+1}\colon A\embed S_{\alpha+1}$ realizes $p_\alpha$.
\[
\xymatrix{
	A \ar[rr]^{i_\alpha}\ar[d]_{e_\alpha} \ar[drr]^{e_{\alpha+1}} & & M_\alpha \ar[d]^{v_\alpha}\\
	S_\alpha \ar[rr]_{s_\alpha} & & S_{\alpha+1}
}
\]

At each limit step $\delta\leq\lambda$, by the $\lambda$-TV, we have $S_\delta=\colim_{\alpha<\delta}S_\alpha\in\K_\lambda$, so we have $s_{\alpha\delta}\colon S_\alpha\embed S_\delta$ for each $\alpha<\delta$.
Then, in the extension $e_\delta:=s_{\alpha\delta}e_\alpha\colon A\embed S_\delta$ (which do not depend on $\alpha<\delta$ as $(s_{\alpha\delta})_{\alpha<\delta}$ is a cocone), the element $s_{\alpha\delta}v_\alpha(c_\alpha)$ realizes $p_\alpha$.
Consequently, $A\embedge{e_\lambda}S_\lambda$ is a pre-saturation.
\end{proof}

\begin{lemma}
Suppose that $\K_\lambda$ satisfies the AP and the $\lambda$-TV, and that $\K$ is Galois $\lambda$-stable.
For every $M_0\in\K_\lambda$, there is a chain $(m_{\alpha\beta}\colon M_\alpha\embed M_\beta)_{\alpha\leq \beta\leq\lambda}$ in $\K_\lambda$ such that:
\begin{enumerate}
\item For every $\alpha<\lambda$, $m_{\alpha,\alpha+1}\colon M_\alpha\embed M_{\alpha+1}$ is a pre-saturation.
\item For every limit ordinal $\delta\leq\lambda$, $M_\delta=\colim_{i<\delta}M_i$.
\end{enumerate}

Such a chain is called a continuous pre-saturation chain.
\end{lemma}

\begin{proof}
Use Lemma \ref{lem:pre-saturation} at each successor stage, and the $\lambda$-TV at each limit stage.
\end{proof}

The following can be viewed as an abstraction of \cite[Lemma 10.5 (1)]{Bal09}.

\begin{lemma}
Suppose that $\K_\lambda$ satisfies the AP and the $\lambda$-TV.
If $(m_{\alpha\beta}\colon M_\alpha\embed M_\beta)_{\alpha\leq\beta\leq\lambda}$ is a continuous pre-saturation chain, then $m_{0\lambda}\colon M_0\embed M_\lambda$ is a universal extension.
\end{lemma}

\begin{proof}
To see that $m_{0\lambda}\colon M_0\embed M_\lambda$ is a universal extension, let $M_0\embedge{e} N$ be an embedding in $\K_\lambda$.
We construct embeddings $(M_i\embedge{e_i}N_i)_{i\leq\lambda}$ and a chain $(N_i\embedge{n_{ij}}N_j)_{i\leq j\leq\lambda}$ in $\K_\lambda$ such that the following diagram commutes:
\[
\xymatrix{
	M_i\ar[rr]^{e_i}\ar[d]_{m_{ij}} & & N_i \ar[d]^{n_{ij}}\\
	M_j\ar[rr]_{e_j} & & N_j
}
\]

First put $N_0=N$, $e_0=e$.

\medskip
\noindent
{\it Successor stage:}
If $N_\beta$ has already been constructed, let $(x^\beta_\gamma)_{\gamma<\lambda}$ be an enumeration of all elements of $N_\beta$.
For any infinite cardinal $\lambda$, we have a bijection $\pair{-,-}\colon\lambda\times\lambda\edge{\simeq}\lambda$ such that $\beta\leq \langle \beta,\gamma\rangle$.

At the $\alpha$th successor stage, if $\alpha$ is of the form $\pair{\beta,\gamma}$ consider $a:=n_{\beta\alpha}(x^\beta_\gamma)\in N_\alpha$.
Let $p$ be the Galois type of $(M_\alpha\embedge{e_\alpha}N_\alpha;a)$.
Since $m_{\alpha,\alpha+1}\colon M_\alpha\embed M_{\alpha+1}$ is a pre-saturation, some $b\in M_{\alpha+1}$ realizes $p$.
That is, $e_\alpha$ and $m_{\alpha,\alpha+1}$ are amalgamated by $n_{\alpha,\alpha+1}$ and $e_{\alpha+1}$ such that $e_{\alpha+1}(b)=n_{\alpha,\alpha+1}(a)=n_{\beta,\alpha+1}(x^\beta_\gamma)$.
\[
\xymatrix{
	M_\alpha\ar[rr]^{e_\alpha}\ar[d]_{m_{\alpha,\alpha+1}} & & N_\alpha \ar[d]^{n_{\alpha,\alpha+1}}\\
	M_{\alpha+1}\ar[rr]_{e_{\alpha+1}} & & N_{\alpha+1}
}
\]

\noindent
{\it Limit stage:}
At a limit stage $\delta$, by the $\lambda$-TV, we get $N_\delta=\colim_{\beta<\delta}N_\beta$.
By the induction hypothesis, we have the following diagram for $\alpha\leq \beta<\lambda$:
\[
\xymatrix{
	M_\alpha\ar[rr]^{e_\alpha}\ar[d]_{m_{\alpha\beta}} & & N_\alpha \ar[d]_{n_{\alpha\beta}} \ar[rr]^{n_{\alpha\delta}} & & N_\delta\\
	M_\beta\ar[rr]_{e_\beta} & & N_\beta \ar[rru]_{n_{\beta\delta}} & &
}
\]
Thus, we get a cocone $(n_{\beta\delta}e_\beta\colon M_\beta\embed N_\delta)_{\beta<\delta}$.
By continuity, we have $M_\delta=\colim_{\beta<\delta}M_\beta$; hence, by the universal property, we obtain $M_\delta\embedge{e_\delta}N_\delta$ such that $e_\delta m_{\beta\delta}=n_{\beta\delta}e_\beta$.

\medskip
\noindent
{\em Verification:}
We claim that the embedding $M_\lambda\embedge{e_\lambda}N_\lambda$ is an isomorphism.
Since $|-|$ is conservative, it suffices to show that $e_\lambda$ is a bijection.
For every $y\in N_\lambda$, by the condition (2) in the $\lambda$-TV axiom, there are $\beta<\lambda$ and $x\in N_\beta$ such that $n_{\beta\lambda}(x)=y$.
Now, $x$ is of the form $x^\beta_\gamma$.
Then put $\alpha=\pair{\beta,\gamma}$.
At the $\alpha$th successor stage, we find some $b\in M_{\alpha+1}$ such that $n_{\beta,\alpha+1}(x^\beta_\gamma)=e_{\alpha+1}(b)$.
Then, we get $y=n_{\beta\lambda}(x^\beta_\gamma)=n_{\alpha+1,\lambda}e_{\alpha+1}(b)=e_\lambda m_{\alpha+1,\lambda}(b)$; that is, for $z:=m_{\alpha+1,\lambda}(b)\in M_\lambda$, we have  $e_\lambda(z)=y$.
This shows that $e_\lambda$ is surjective.
Since an embedding is always injective, this verifies the claim.

Then, consider $g:=e^{-1}_\lambda n_{0,\lambda}\colon N_0\embed M_\lambda$.
Since $n_{0\lambda}e_0=e_\lambda m_{0\lambda}$, we get 
$g\circ e=e^{-1}_\lambda n_{0\lambda} e_0=m_{0\lambda}$.
This shows that $m_{0\lambda}\colon M_0\embed M_\lambda$ is a universal extension.
\end{proof}

This concludes the proof of Theorem \ref{thm:Shelah-AEC}.
Combining Theorem \ref{thm:choice-vs-universal-extension} with this yields Theorem \ref{thm:main-stability-thm} (1)$\Rightarrow$(3).


\section{The internal axiom of choice}\label{sec:IAC-structure}

In this section, we characterize the internal axiom of choice in terms of the existence of maximal structures.

\begin{definition}
A structure $A\in\K$ is maximal (in $\K$) if, for every extension $A\embedge{f}B\in\K$, $f$ is an isomorphism.

A structure-class $\K_\lambda$ has the {\bf maximal extension property} if, for every $A\in\K_\lambda$, there exists an embedding $A\embed M\in\K_\lambda$ into a maximal structure (in $\K_\lambda$).
\end{definition}

\begin{theorem}\label{thm:IAC-maximalcofinal}
Suppose that $\K_\lambda$ is a subcanonical AP class.
Then the following are equivalent:
\begin{enumerate}
\item $\K_\lambda$ has the maximal extension property.
\item ${\rm Sh}(\K_\lambda^{\rm op},J_{\rm at})$ satisfies the internal axiom of choice.
\end{enumerate}
\end{theorem}

We prepare for the proof.

\subsection{Maximal structures and projective objects}

\begin{definition}
An object $P$ is projective if, for every epimorphism $f\colon X\surj Y$ and every morphism $P\edge{y}Y$, there exists a morphism $P\edge{x}X$ such that $fx=y$.
\end{definition}

The external axiom of choice asserts that every object is projective, while the internal axiom of choice asserts that every object is internally projective.

\begin{lemma}\label{lem:maximal-equ-projec}
Suppose that $\K_\lambda$ is a subcanonical AP class.
Then the following are equivalent:
\begin{enumerate}
\item $M$ is a maximal structure in $\K_\lambda$.
\item $\yo M$ is projective.
\end{enumerate}
\end{lemma}

\begin{proof}
(1)$\Rightarrow$(2)
Suppose that $M$ is maximal.
Take any $f\colon X\surj Y$ and $y\colon\yo M\to Y$.
By the Yoneda lemma, the latter corresponds to an element $y\in Y_M$.
Since epimorphisms are locally surjective, there exist a cover $M\embedge{c} N$ and an element $x\in X_N$ such that $f_N(x)=y\at c$.
By maximality of $M$, $c$ is an isomorphism.
Considering $c^{-1}\colon N\embed M$, we have $x\at c^{-1}\in X_M$ and $f_M(x\at {c^{-1}})=f_N(x)\at {c^{-1}}=(y\at c)\at {c^{-1}}=y$.
The element $x\at {c^{-1}}\in X_M$ may be identified with a morphism $x\colon\yo M\to X$, and $fx=y$.
Thus $\yo M$ is projective.

(2)$\Rightarrow$(1)
Suppose that $\yo M$ is projective, and take any $M\embedge{i} N$.
By subcanonicity, Lemma \ref{lem:atomic-topos-embedding-to-epi} implies that $i':=\yo i\colon\yo N\surj \yo M$ is an epimorphism.
Since $\yo M$ is projective, there exists a section $\yo M\edge{s'}\yo N$.
By subcanonicity and Yoneda's lemma, this corresponds to a morphism $N\embedge{s} M$.
Indeed, $i's'=1$ implies $si=1$.
Moreover, $s$ is monic, and $s(is)=(si)s=s$, so $is=1$.
Thus $i$ is an isomorphism.
Since $i$ was arbitrary, $M$ is maximal.
\end{proof}


\begin{lemma}\label{lem:maximal-projective-atom}
Suppose that $\K_\lambda$ is a subcanonical AP class.
For every projective atom $P$, there exists a maximal $M\in\K_\lambda$ such that $P\simeq\yo M$.
\end{lemma}

\begin{proof}
By subcanonicity and Observation \ref{obs:yo-generator}, $\{\yo A:A\in\K_\lambda\}$ is a generating family consisting of atoms.
For every atom $P$, Observation \ref{obs:yo-generator} (4) gives a structure $M\in\K_\lambda$ and an epimorphism $\yo M\to P$.
If $P$ is projective, there is a section $P\monoedge{s}\yo M$.
Since $P$ and $\yo M$ are both atoms, $s$ is an isomorphism.
Therefore, $P\simeq\yo M$.
Since $P$ is projective, Lemma \ref{lem:maximal-equ-projec} implies that $M$ is maximal.
\end{proof}

\subsection{The atomic presentation axiom}
The presentation axiom $\textsf{PAx}$, also called ${\sf CoSHEP}$ or ${\sf EPC}$, is a weakening of the external axiom of choice \cite{Acz78,Bla79}.
It asserts that there are enough projectives: for every object $A$, there exists an epimorphism $P\surj A$ from a projective object.

\begin{prop}
Let $\mathcal{E}$ be an atomic Grothendieck topos.
Then, the following are equivalent:
\begin{enumerate}
\item $\mathcal{E}$ satisfies the presentation axiom $\mathsf{PAx}$.
\item $\mathcal{E}$ satisfies the atomic presentation axiom $\mathsf{AtPAx}$: for every atom $A$, there exists an epimorphism $P\surj A$ from a projective atom $P$.
\end{enumerate}
\end{prop}

\begin{proof}
(1) Let $A$ be an atom.
By ${\sf PAx}$, there is an epimorphism $e\colon P\surj A$ from a projective object $P$.
By Observation \ref{obs:yo-generator} (3), decompose $P$ into atoms $P\simeq\sum_{j\in J}P_j$.
One may see that each $P_j$ is also projective.
Since $P_j\not\simeq 0$ and $A$ is an atom, $e\circ\iota_j\colon P_j\to A$ is an epimorphism from a projective atom.

(2)$\Rightarrow$(1):
By Observation \ref{obs:yo-generator} (3), every object $X$ can be decomposed as $X\simeq \sum_{i\in I} A_i$.
For each atom $A_i$, by ${\sf AtPAx}$, we have an epimorphism $p_i\colon P_i\surj A_i$ from a projective atom $P_i$.
Then $\sum_i p_i\colon \sum_i P_i\to \sum_i A_i\simeq X$ is an epimorphism, where one may check that $\sum_iP_i$ is a projective.
\end{proof}

\begin{theorem}\label{thm:maximal-cofinal-theorem}
Suppose that $\K_\lambda$ is a subcanonical AP class.
Then the following are equivalent.
\begin{enumerate}
\item $\K_\lambda$ has the maximal extension property.
\item ${\rm Sh}(K_\lambda^{\rm op},J_{\rm at})$ satisfies the atomic presentation axiom ${\sf AtPAx}$.
\end{enumerate}
\end{theorem}

\begin{proof}
(1)$\Rightarrow$(2)
We prove the atomic presentation axiom.
By Observation \ref{obs:yo-generator} (4), for every atom $A$ there exists an epimorphism $\yo N\surj A$.
By the maximal extension property, there is a morphism $N\embed M$ into a maximal $M$.
By Lemma \ref{lem:atomic-topos-embedding-to-epi}, this induces an epimorphism $\yo M\surj \yo N$.
By Lemma \ref{lem:maximal-equ-projec} (1)$\Rightarrow$(2), $\yo M$ is a projective atom.
Thus the composite epimorphism $\yo M\surj\yo N\surj A$ from a projective atom establishes the atomic presentation axiom.

(2)$\Rightarrow$(1)
Take any structure $N\in\K_\lambda$.
By the atomic presentation axiom, there exists a morphism $P\surj \yo N$ from some projective atom $P$.
By Lemma \ref{lem:maximal-projective-atom}, $P$ is of the form $\yo M$ for some maximal $M$.
By subcanonicity, $\yo M\to\yo N$ corresponds to a morphism $N\to M$.
\end{proof}


\subsection{${\sf AtPAx}=$ the internal axiom of choice}

\begin{theorem}\label{thm:apPAx-to-IAC}
The following are equivalent:
\begin{enumerate}
\item ${\rm Sh}(\C,J_{\rm at})$ satisfies the atomic presentation axiom ${\sf AtPAx}$.
\item ${\rm Sh}(\C,J_{\rm at})$ satisfies the internal axiom of choice.
\end{enumerate}
\end{theorem}

We prepare to relate the atomic presentation axiom to the internal axiom of choice.
In the following, we work in the atomic sheaf topos ${\rm Sh}(\C,J_{\rm at})$.
We do not assume subcanonicity.

\begin{lemma}\label{lem:proj-prod}
If $P$ is a projective atom, then $P\times Z$ is projective for every $Z$.
\end{lemma}

\begin{proof}
By Observation \ref{obs:yo-generator}, there is an atomic decomposition $P\times Z\simeq\sum_{i\in I}A_i$.
For each $A_i$, consider the projection $f\colon A_i\embed\sum_{i\in I}A_i\surj P$.
Since this is a morphism between atoms, it is an epimorphism by Observation \ref{obs:atomic-epi-nontrivial} (1).

Since $P$ is projective, for $1_P\colon P\to P$ there exists a morphism $P\edge{x} A_i$ such that $fx=1_P$.
Now $x$ is a section of $f$, hence is monic; since $A_i$ and $P$ are atom (so $P\not=\mathbf{0}$), we get $A_i\simeq P$.
Therefore, $P\times Z\simeq\sum_{i\in I}P$.
A coproduct of projective objects is projective, so $P\times Z$ is projective.
\end{proof}

\begin{lemma}\label{lem:atomic-projective}
For an atom $P$, the following are equivalent:
\begin{enumerate}
\item $P$ is projective.
\item For every epimorphism $f\colon A\surj B$ between atoms and every morphism $P\edge{y}B$, there exists a morphism $P\edge{x}A$ such that $fx=y$.
\end{enumerate}
\end{lemma}

\begin{proof}
(1)$\Rightarrow$(2):
Immediate.

(2)$\Rightarrow$(1):
Take any $f\colon X\surj Y$ and $P\edge{y}Y$.
Since $P$ is an atom, Observation \ref{obs:atomic-epi-nontrivial} (2) implies that ${\rm Im}(y)\embed Y$ is also an atom.
Consider the pullback $\pi\colon X\times_Y{\rm Im}(y)\surj {\rm Im}(y)$.
By Observation \ref{obs:yo-generator}, in an atomic sheaf topos its domain decomposes as a coproduct $\sum_{i\in I}A_i$ of atoms.
Since $\pi$ is an epimorphism, $I\not=\emptyset$.
Hence there exists a morphism $A_i\edge{b}{\rm Im}(y)$ from an atom.
By Observation \ref{obs:atomic-epi-nontrivial} (1), this is an epimorphism.

By the assumed projectivity with respect to atoms, applied to the corestriction $\bar{y}\colon P\surj {\rm Im}(y)$ of $y$ and to $A_i\edge{b} {\rm Im}(y)$, there exists $P\edge{a}A_i$ such that $ba=\bar{y}$.
Composing, we obtain $x\colon P\edge{a}A_i\embed \sum_{i\in I}A_i\surj X$, and this satisfies $fx=y$.
\end{proof}

\begin{definition}
For an epimorphism $f\colon A\surj B$ between atoms, define ${\rm Sec}(f)$ by the following pullback.
\[
\xymatrix{
{\rm Sec}(f) \ar[rr]\ar[d] & & A^B \ar[d]^{f^B} \\
1 \ar[rr] & & B^B 
}
\]

A morphism $T\to {\rm Sec}(f)$ then corresponds to a morphism $T\times B\edge{x} A$ satisfying $fx=\pi_B$.
In terms of sheaves, this says explicitly that $f_px_p(t,b)=b$ for every $p\in\C$.
\end{definition}

The external axiom of choice asserts that if $f$ is an epimorphism, then $f$ admits a section.

\begin{lemma}\label{lem:sec-object-epi}
Assume the internal axiom of choice.
For every epimorphism $f\colon A\surj B$ between atoms, ${\rm Sec}(f)\to 1$ is also an epimorphism.
\end{lemma}

\begin{proof}
Since $f$ is an epimorphism, the internal axiom of choice implies that $f^B\colon A^B\surj B^B$ is also an epimorphism.
Pullbacks of epimorphisms are epimorphisms, so ${\rm Sec}(f)\surj 1$ is an epimorphism.
\end{proof}

\begin{proof}[Proof of Theorem \ref{thm:apPAx-to-IAC}]
(1)$\Rightarrow$(2):
To prove the internal axiom of choice, take an epimorphism $f\colon X\surj Y$ and an object $Z$.
We shall show that $f^Z\colon X^Z\to Y^Z$ is an epimorphism.

By Observation \ref{obs:yo-generator}, decompose $Y^Z$ into atoms and choose an atom $A\embedge{i} Y^Z$.
By the atomic presentation axiom, there exists an epimorphism $r\colon P\surj A$ from a projective atom.
Uncurrying the composite $\hat{y}\colon P\edge{r}A\edge{i}Y^Z$ gives a morphism $P\times Z\edge{y} Y$.
By Lemma \ref{lem:proj-prod}, $P\times Z$ is projective, so there exists a morphism $P\times Z\edge{x}X$ such that $fx=y$.
Currying it yields a morphism $P\edge{\hat{x}}X^Z$.
Then $f^Z\hat{x}=\hat{y}=r\circ i$.

Now $r\circ i\colon P\surj A\mono Y^Z$ is an image factorization. Comparing it with the image factorization $X^Z\surj {\rm Im}(f^Z)\mono Y^Z$ of $f^Z$, we see that $A\mono Y^Z$ is a subobject of ${\rm Im}(f^Z)\mono Y^Z$.
Thus, for the atomic decomposition $Y^Z=\sum_{i\in I}A_i$, every $A_i$ is a subobject of ${\rm Im}(f^Z)$, and therefore ${\rm Im}(f^Z)\simeq Y^Z$.
Hence $f^Z$ is an epimorphism.

(2)$\Rightarrow$(1):
Fix an atom $A$.
By Observation \ref{obs:yo-generator} (4), every atom is a quotient of an atom of the form $\yo p$, so the isomorphism classes of atoms may be taken to form a set.
We may therefore choose a family of representatives $\{f_i\colon B_i\surj C_i\}_{i\in I}$ of all epimorphisms between atoms.
By Lemma \ref{lem:sec-object-epi}, each ${\rm Sec}(f_i)\surj 1$ is an epimorphism; by the axiom of choice for $I$-indexed families, $S:=\prod_{i\in I}{\rm Sec}(f_i)\surj 1$ is also an epimorphism.
Then its pullback $A\times S\surj A$ is an epimorphism, so we get $A\times S\not\simeq \mathbf{0}$,

Decompose $A\times S\surj A$ into atoms, choose one of them, and denote it by $P$.
We then obtain $r\colon P\embed\sum_{i\in I}P_i\simeq A\times S\surj A$.
Since this is a morphism between atoms, Observation \ref{obs:atomic-epi-nontrivial} (1) implies that $r\colon P\surj A$ is an epimorphism.
There is also a morphism $P\to A\times S\to S=\prod_{i\in I}{\rm Sec}(f_i)$, which consequently yields each component $P\to {\rm Sec}(f_i)$.
This corresponds to a parametrized section $x_i\colon P\times C_i\to B_i$, satisfying $f_ix_i=\pi_{C_i}$.

We show that $P$ is projective with respect to atoms.
Given an epimorphism $f_i\colon B_i\surj C_i$ between atoms and any morphism $P\edge{y} C_i$, consider $P\xedge{\pair{1,y}}P\times C_i\edge{x_i}B_i$. Then $f_ix_i\pair{1,y}=y$.
Thus $P$ is projective with respect to atoms.
By Lemma \ref{lem:atomic-projective}, $P$ is in fact projective.
We have therefore obtained an epimorphism $p\colon P\surj A$ from a projective object.
\end{proof}

\begin{proof}[Proof of Theorem \ref{thm:IAC-maximalcofinal}]
This follows from Theorem \ref{thm:maximal-cofinal-theorem} and Theorem \ref{thm:apPAx-to-IAC}.
\end{proof}

\subsection{The groupoid of maximal structures}\label{sec:groupoid}

Theorem \ref{thm:IAC-maximalcofinal} (1)$\Rightarrow$(2) can be explained in another way, which has the advantage of not assuming subcanonicity.
Below we abbreviate $\K=\K_\lambda$.
Let $\mathcal{K}_{\rm max}\subseteq\mathcal{K}$ be the collection of all maximal structures in $\mathcal{K}$.
We regard $\mathcal{K}_{\rm max}$ as a full subcategory of $\K$.

\begin{obs}\label{obs:groupoid-topology}~
\begin{enumerate}
\item $\mathcal{K}_{\rm max}$ is a groupoid.
\item Every presheaf on a groupoid $G$ is a $J_{\rm at}$-sheaf.
\end{enumerate}
\end{obs}

\begin{proof}
(1)
By maximality, every morphism $M\embed N$ with $M,N\in\mathcal{K}_{\rm max}$ is an isomorphism.
(2)
In a groupoid, every $J_{\rm at}$-cover $q\to p$ is an isomorphism. Thus the atomic topology is trivial, so $J_{\rm at}$-sheaves are precisely presheaves.
\end{proof}

Consider the full subcategory $[M]\subseteq\K_{\max}$ consisting of the isomorphism class of a maximal structure $M\in\K_{\max}$.
This is a connected groupoid.
\begin{obs}[see e.g.~{\cite[Proposition 1.5.13]{Riehlbook}}]\label{obs:Riehl}
The category $[M]$ is equivalent to the automorphism group ${\rm Aut}(M)$ regarded as a one-object category.
\end{obs}

\begin{proof}
For each $N\in[M]$, choose an isomorphism $i_N\colon M\edge{\simeq} N$.
Define a functor $[M]\to{\rm Aut}(M)$ by sending every object $N$ to the object $\ast$ and every morphism $L\edge{f}N$ to $F(f):=i_N^{-1}\circ f\circ i_L\colon M\to M$.
It is straightforward to verify that this is an equivalence of categories.
\end{proof}

Choosing a representative $M_i$ of each isomorphism class in $\K_{\max}$, we may write
\[\K_{\max}=\sum_{i\in I}M_i\simeq\sum_{i\in I}{\rm Aut}(M_i)\]
Moreover, by Observation \ref{obs:groupoid-topology}, $J_{\rm at}$-sheaves on a groupoid are precisely presheaves, and hence
\[
{\rm Sh}(\K_{\rm max}^{\rm op},J_{\rm at})\simeq\prod_{i\in I}{\rm Aut}(M_i)\text{-}{\bf Set}.
\]

A full subcategory $\D\embed\C$ is $J$-dense if, for every $p\in\C$, $\{q\edge{f}p:q\in\D\}$ is a $J$-cover of $p$.
In particular, when $J=J_{\rm at}$, this means that for every $p\in \C$ there exists a morphism $q\to p$ from some $q\in \D$.

\begin{obs}
The following are equivalent:
\begin{enumerate}
\item $\K$ has the maximal extension property.
\item $\K^{\rm op}_{\rm max}\embed\K^{\rm op}$ is $J_{\rm at}$-dense.
\end{enumerate}
\end{obs}

\begin{proof}
Compare the definitions.
\end{proof}

\begin{lemma}[Comparison Lemma (see e.g.~{\cite[p.~590]{SGL}} and {\cite[C2.2.3]{Elephant2}}]
If $\D$ is a $J_{\rm at}$-dense full subcategory of $\C$, then ${\rm Sh}(\C,J_{\rm at})\simeq{\rm Sh}(\D,J_{at})$.
\end{lemma}

\begin{cor}
If $\K$ has the maximal extension property, then ${\rm Sh}(\K^{\rm op},J_{\rm at})\simeq{\rm Sh}(\K_{\rm max}^{\rm op},J_{\rm at})$.
\end{cor}

Thus the maximal extension property says that the behavior of atomic sheaves on $\K$ depends only on the maximal structures.
Consequently, if $\K$ has the maximal extension property, then
\begin{align}\label{equ:at-sheaf-maximal-automorphism}
{\rm Sh}(\K^{\rm op},J_{\rm at})\simeq\prod_{i\in I}{\rm Aut}(M_i)\text{-}{\bf Set}
\end{align}
Moreover, the topos on the right-hand side satisfies the internal axiom of choice in general \cite{Fre,FrSc90}.
In fact, the right-hand side is a typical example of a Boolean \'etendue \cite{FrSc90}.
This explains Theorem \ref{thm:IAC-maximalcofinal} (1)$\Rightarrow$(2).

\subsection{The external axiom of choice}

We also analyze when the external axiom of choice holds.

\begin{definition}
A structure $M\in\mathcal{K}$ is {\bf rigid} if ${\rm Aut}_\K(M)=1$; that is, its only automorphism $M\embed M$ is the identity ${\rm id}_M$.
\end{definition}

\begin{theorem}\label{thm:EAC-main-thm}
Suppose that $\K_\lambda$ is a subcanonical AP class.
Then the following are equivalent:
\begin{enumerate}
\item $\K_\lambda$ has the maximal extension property, and every maximal structure is rigid.
\item ${\rm Sh}(\K_\lambda^{\rm op},J_{\rm at})$ satisfies the external axiom of choice.
\end{enumerate}
\end{theorem}

We prepare for the proof.

%

\begin{lemma}\label{lem:Gset-EAC}
The presheaf category on a group $G$ satisfies the external axiom of choice if and only if $G=1$.
\end{lemma}

\begin{proof}
This is well known, but we provide a proof to make it self-contained.

($\Leftarrow$)
If $G=1$, then $G\text{-}{\bf Set}={\bf Set}$, so the external axiom of choice follows from the axiom of choice in the metatheory.

($\Rightarrow$)
Let $\ast$ be the unique object of the one-object category $G$. Then $\yo\ast$ is $G$ itself equipped with the right action $x\at g=xg$.
We denote it simply by $G$.
The morphism $p\colon G\to 1$ is surjective as a function, hence is an epimorphism.
If a section $s\colon 1\to G$ exists and $x:=s(\ast)\in G$, then for every $g\in G$ we have $x=s(\ast)=s(\ast\at g)=s(\ast)\at g=x\at g=xg$.
Multiplying both sides on the left by $x^{-1}$ gives $e=g$.
Therefore $G=1$.
\end{proof}

\begin{proof}[Proof of Theorem \ref{thm:EAC-main-thm}]
The external axiom of choice implies the internal axiom of choice, so Theorem \ref{thm:IAC-maximalcofinal} gives the maximal extension property.
Thus, under either (1) or (2), the equation (\ref{equ:at-sheaf-maximal-automorphism}) from Section \ref{sec:groupoid} holds.

(1)$\Rightarrow$(2):
Since each maximal structure $M_i$ is rigid, ${\rm Aut}(M_i)\simeq 1$, and the right-hand side is ${\bf Set}^I$.
By the axiom of choice in the metatheory, this clearly satisfies the external axiom of choice.

(2)$\Rightarrow$(1):
The right-hand side satisfies the external axiom of choice only if ${\rm Aut}(M_i)\text{-}{\bf Set}$ does so for every $i\in I$. By Lemma \ref{lem:Gset-EAC}, this occurs only when ${\rm Aut}(M_i)\simeq 1$.
Therefore, every maximal structure is rigid.
\end{proof}

\section{Concrete examples}

\subsection{Elementary examples (with elementary embeddings)}

Important examples to which our results apply are first-order classes.

\begin{fact}\label{fact:elementary-class}
For every first-order theory $T$, the class $\K={\rm Mod}(T)$ of all models of $T$ with elementary embeddings is a coherent structure-class satisfying the strong AP and $\lambda$-TV.
\end{fact}

\begin{proof}
For the strong AP, see e.g.~\cite[Section 6.4.3]{Hod93}.
For $\lambda$-TV, it is the Tarski--Vaught elementary chain argument~\cite[Proposition 2.3.11]{Mar02}.
\end{proof}

\begin{remark}
When morphisms are embeddings (not necessarily elementary), the AP and $\lambda$-TV may not hold, so it is necessary to verify this on a case-by-case basis.
If a first-order theory $T$ is model-complete \cite[Definition 3.1.13]{Mar02}, every embedding is an elementary embedding, so it is automatic.
\end{remark}

By Proposition \ref{prop:stAP-subcanonical} and Fact \ref{fact:elementary-class}, every such first-order class with elementary embeddings yields a subcanonical atomic site.
Therefore, all results in Sections \ref{sec:choice-stable} and \ref{sec:IAC-structure} hold for $\K={\rm Mod}(T)$.
In this case, since Galois types are simply ordinary complete types, Galois $\lambda$-stability corresponds to the classical model-theoretic notion of being stable in $\lambda$.
Hence, Theorem \ref{thm:main-stability-thm} implies:
\[
\text{${\rm Sh}({\rm Mod}(T)_\lambda^{\rm op},J_{\rm at})$ satisfies AC for $\lambda^+$-families}
\iff
\text{$T$ is stable in $\lambda$.}
\]

Let ${\rm LS}(T)$ be the least infinite cardinal greater than or equal to the cardinality of the language of $T$.
The {\bf stability spectrum} of a first-order theory $T$ is the class of cardinals $\lambda\geq {\rm LS}(T)$ such that $T$ is stable in $\lambda$.
Then the first stable cardinal $\lambda(T)$ is the least cardinal in the stability spectrum if it exists.

\begin{fact}[Shelah's Stability Spectrum Theorem {\cite[Theorem III.5.15]{She90}}]
If a complete first-order theory $T$ is stable, then there exists a cardinal $\kappa(T)$ such that:
\[
\text{$T$ is stable in $\lambda$}
\iff 
\text{$\lambda\geq\lambda(T)$ and $\lambda^{<\kappa(T)}=\lambda$}.
\]
\end{fact}

\begin{cor}
Let $T$ be a stable complete first-order theory, and let $\lambda$ be an infinite cardinal greater than or equal to the language of $T$.
Then, the following are equivalent.
\begin{enumerate}
\item ${\rm Sh}({\rm Mod}(T)_\lambda^{\rm op},J_{\rm at})$ satisfies the axiom of choice for $\lambda^+$-indexed families.
\item $\lambda\geq\lambda(T)$ and $\lambda^{<\kappa(T)}=\lambda$.
\end{enumerate}
\end{cor}

For the calculation of the values of $\kappa$ and $\lambda$, see, for example, \cite[Theorem 8.6.5]{TeZi12}.
From this, for example, we can obtain the following.

\begin{example}\label{exa:stable-spectra}
Let $T$ be a countable complete first-order theory.
\begin{enumerate}
\item If $T$ is $\om$-stable, then $\kappa(T)=\lambda(T)=\aleph_0$.
Thus:
\[\text{${\rm Sh}({\rm Mod}(T)_\lambda^{\rm op},J_{\rm at})$ satisfies AC for every (external-)set indexed families.}\]
\item If $T$ is superstable, but not $\om$-stable, then $\kappa(T)=\aleph_0$.
Thus: 
\[\text{${\rm Sh}({\rm Mod}(T)_\lambda^{\rm op},J_{\rm at})$ satisfies AC for $\lambda^+$-families $\iff$ $\lambda\geq\lambda(T)$}.\]
\item If $T$ is stable, but not superstable, then $\kappa(T)=\aleph_1$.
Thus:
\[\text{${\rm Sh}({\rm Mod}(T)_\lambda^{\rm op},J_{\rm at})$ satisfies AC for $\lambda^+$-families $\iff$ $\lambda^\om=\lambda$.}\]
\item If $T$ is unstable, for any infinite cardinal $\lambda$,
\[\text{${\rm Sh}({\rm Mod}(T)_\lambda^{\rm op},J_{\rm at})$ does not satisfy AC for $\lambda^+$-families}.\]
\end{enumerate}
\end{example}

\subsection{Elementary examples (with embeddings)}\label{sec:elementary-example-embedding}

We now explain the entries in Table \ref{table-1}, with the exception of well-orders.
As noted above, Galois types and quantifier-free types coincide for many classes $\K$ of structures \cite[Remark 3.8]{Vas17}.
Let $S(\K,\lambda;A)$ be the set of all Galois types over $A$ in $\K_\lambda$.
Then the Galois type spectrum $s(\K,\lambda)$ is defined as 
\[\sup\{\mbox{the cardinality of }S(\K,\lambda;A):A\in\K_\lambda\}.\]

Keisler \cite{Kei76} introduced a notion essentially equivalent to that of a quantifier-free type and computed the type spectrum in a variety of concrete examples.
\begin{example}[see e.g.~Keisler \cite{Kei76}]
Let $\lambda$ be an infinite cardinal.
\begin{enumerate}
\item Let ${\rm Grp}$ be the class of all groups.
Then $s({\rm Grp};\lambda)=2^\lambda$.
\item Let ${\rm AbGrp}$ be the class of all abelian groups.
Then $s({\rm AbGrp};\lambda)=\lambda$.
\item Let ${\rm LO}$ be the class of all linear orders.
Then $s({\rm LO};\lambda)>\lambda$.
\item Let ${\rm BA}$ be the class of all Boolean algebras.
Then $s({\rm BA};\lambda)=2^\lambda$.
\item Let ${\rm Vec}_F$ be the class of all vector spaces over a field $F$.
Then $s({\rm Vec}_F;\lambda)=\lambda$ (if $\lambda$ is larger than or equal to the cardinality of $F$).
\end{enumerate}
\end{example}

All the classes of structures above have the strong AP and colimits of $\lambda$-chains.
\begin{enumerate}
\item Dependent choice: Since these classes have colimits of chains, Theorem \ref{lem:cone-characterization-dependent-choice} implies that they satisfy the axiom of dependent choice.
\item The axiom of choice for $\lambda$-indexed families:
Suppose a fan of size $\lambda$ is given.
By successively amalgamating the fan at successor stages of a transfinite recursion and taking colimits at limit stages, one obtains a cocone over the fan.
Thus, by Theorem \ref{lem:cone-characterization-countable-choice}, the axiom of choice for $\lambda$-indexed families holds.
\item The axiom of choice for $\lambda^+$-indexed families:
Among these examples, only sets, vector spaces, and abelian groups are stable; the other classes are unstable.
Thus Theorem \ref{thm:main-stability-thm} determines  for which of these examples the axiom of choice for $\lambda^+$-indexed families holds.
\item The internal axiom of choice:
None of sets, vector spaces, or abelian groups has the maximal extension property.
For sets, consider $X\embed X+1$; for $F$-vector spaces, consider $V\embed V\oplus F$; and for abelian groups, consider $A\embed A\oplus\mathbb{Z}$.
Hence, by Theorem \ref{thm:IAC-maximalcofinal}, the internal axiom of choice does not hold.
\end{enumerate}

The items (1), (2) and (4) hold whether morphisms are taken to be embeddings or elementary embeddings.
Note that the argument in the item (2) actually shows the following:
\begin{obs}
If a structure-class $\K$ satisfies the AP and the $\lambda$-TV, then ${\rm Sh}(\K_\lambda^{\rm op},J_{\rm at})$ always satisfies the axiom of choice for $\lambda$-families.
\end{obs}

\subsection{Non-elementary example}

It is easy to verify that the class ${\rm WO}$ of well-orders has the strong AP.
A crucial difference from the preceding examples, however, is that it has no colimits of chains (indeed, some chains do not even admit cocones).
This yields the following result.

\begin{prop}
Let $\lambda$ be an infinite cardinal.
$\mathrm{Sh}(\mathrm{WO}^{\rm op}_\lambda,J_{\rm at})$ does not satisfy the internal axiom of dependent choice.
\end{prop}

\begin{proof}
By subcanonicity and Theorem \ref{lem:cone-characterization-dependent-choice}, it suffices to show that a sequence of morphisms $\{\alpha_n\embed\alpha_{n+1}\}$ in the category $\mathrm{WO}_\lambda$ need not admit a cocone.

Define an embedding $\om\embedge{s}\om$ by $s(n)=n+1$.
Suppose that the sequence $\om\embedge{s}\om\embedge{s}\om\embedge{s}\cdots$ admits a cocone $e_n\colon\om\to\delta$.
By the definition of a cocone, $e_{n+1}s=e_n$.
In particular, $e_{n+1}(0)<e_{n+1}(1)=e_n(0)$.
This gives an infinite descending sequence $e_0(0)>e_1(0)>e_2(0)>\cdots$ in $\delta$.
\end{proof}

To determine the validity of the axiom of choice, one might examine stability.
\begin{obs}
For every infinite cardinal $\lambda$, ${\rm WO}$ is Galois $\lambda$-stable.
\end{obs}

\begin{proof}
For a pointed extension $(\alpha\embedge{i}\beta;\gamma)$ of ordinals of cardinality at most $\lambda$, it is easy to verify that its Galois type is determined by the left cut $L_i(\gamma)=\{\xi<\alpha:i(\xi)<\gamma\}$ together with the truth value of $\gamma\in i[\alpha]$.
However, the initial segments of an ordinal $\alpha$ are precisely $\{\beta:\beta\leq\alpha\}$, whose cardinality is $\lambda$.
Thus $s({\rm WO},\lambda)=\lambda$, and hence ${\rm WO}$ is Galois $\lambda$-stable.
\end{proof}

However, $\mathrm{WO}$ is not an AEC, so Shelah's result (Fact \ref{fact:Shelah-AEC}) does not apply.
To be explicit, $\mathrm{WO}$ does not satisfy $\lambda$-TV.
Consequently, one cannot infer from this the axiom of choice for every external cardinal.
Indeed, this gives an interesting counterexample in which stability holds but the relevant axiom of choice fails.

\begin{prop}
${\rm Sh}({\rm WO}_\lambda^{\rm op},J_{\rm at})$ does not satisfy the axiom of choice for $\lambda^+$-indexed families.
\end{prop}

\begin{proof}
Consider the family of inclusion maps $\{\lambda\embed\lambda+\gamma+1\}_{\gamma<\lambda^+}$.
By Theorem \ref{lem:cone-characterization-countable-choice}, it suffices to show that this family admits no cocone in ${\rm WO}_\lambda$.
Suppose a cocone $\{\lambda+\gamma+1\embed\delta\}$ exists. Since $\delta\in{\rm WO}_\lambda$, the cardinality of $\delta$ is at most $\lambda$.
On the other hand, every ordinal below $\lambda^+$ embeds into $\delta$, which implies $\delta\geq\lambda^+$.
\end{proof}

Since cocones over chains need not exist, the validity of the axiom of choice must be verified by a different argument.

\begin{prop}
${\rm Sh}({\rm WO}_\lambda^{\rm op},J_{\rm at})$ satisfies the axiom of choice for $\lambda$-indexed families.
\end{prop}

\begin{proof}
Let $\alpha$ and $\beta$ be ordinals of cardinality at most $\lambda$.
An embedding $\alpha\embedge{i}\beta$ determines a decomposition of $\beta$ into $\alpha+1$ blocks $B_\xi$ by cutting $\beta$ at each point $i(\xi)\in\beta$.
More precisely, for each $\xi\leq \alpha$, let $A_\xi=\{x\in\beta\setminus i[\alpha]:x<i(\xi)\}$ and define $B_\xi=A_\xi\setminus\bigcup_{\zeta<\xi}A_\zeta$.
Here we set $A_\alpha=\beta\setminus i[\alpha]$.
Thus $\beta$ decomposes as follows.
\[
\beta\simeq\sum_{\xi<\alpha}(B_\xi+1)+B_\alpha
\]

By Theorem \ref{lem:cone-characterization-countable-choice}, it suffices to show that every $\lambda$-indexed family of morphisms $\{\alpha\embed \beta_\eta\}_{\eta<\lambda}$ admits a cocone.
For this purpose, suppose that each embedding $i_\eta\colon\alpha\embed\beta_\eta$ yields a decomposition $\beta_\eta\simeq\sum_{\xi<\alpha}(B_\xi^\eta+1)+B_\alpha^\eta$.
Define a well-order $\gamma$ by
\[
\gamma=\sum_{\xi<\alpha}\left(\sum_{\eta<\lambda}B^\eta_\xi+1\right)+\sum_{\eta<\lambda}B^\eta_\alpha.
\]

This identifies $(i_\eta(\xi))_{\eta<\lambda}$ with $1$ for each $\xi<\alpha$, while preserving, in order, copies of $\beta_\eta\setminus i[\alpha]$.
The cardinality of $\gamma$ is at most that of $\alpha\cdot(\lambda+1)+\lambda$.
\end{proof}

%
%
%

\subsection{Examples for the internal AC}\label{exa:IAC}

By Theorem \ref{thm:IAC-maximalcofinal}, the maximal extension property is necessary for the internal AC.
Most familiar classes of structures, however, have no maximal structures at all; indeed, studies of AECs often explicitly assume the absence of maximal models.

One simple way to guarantee the existence of maximal structures is to impose a finite bound on cardinality or dimension.
For example, consider the class ${\rm Set}_{\leq n}$ of sets of cardinality at most $n$, or the class ${\rm Vec}_{F,\leq n}$ of vector spaces of dimension at most $n$, and both classes have the maximal extension property.
Their maximal structures are, respectively, the $n$-element sets and the $n$-dimensional vector spaces.
%

Consider the groupoids ${\rm Set}_{=n}$ of all $n$-element sets and ${\rm Vec}_{F,=n}$ of all $n$-dimensional vector spaces. By Observation \ref{obs:Riehl}, the former is equivalent to the symmetric group $S_n$ regarded as a one-object category, and the latter to the general linear group ${\rm GL}_n(F)$ regarded as a one-object category.
Thus, by the discussion in Section \ref{sec:groupoid},
\begin{align*}
{\rm Sh}({\rm Set}_{\leq n}^{\rm op},J_{\rm at})\simeq S_n\text{-}{\bf Set},
& &
{\rm Sh}({\rm Vec}_{F,\leq n}^{\rm op},J_{\rm at})\simeq {\rm GL}_n(F)\text{-}{\bf Set}
\end{align*}
and these toposes satisfy the internal but not the external axiom of choice (if $n>1$).

\subsection{Examples for the external AC}

By the equation (\ref{equ:at-sheaf-maximal-automorphism}) in Section \ref{sec:groupoid}, and Theorem \ref{thm:EAC-main-thm}, ${\rm Sh}(\K^{\rm op},J_{\rm at})$ satisfies the external axiom of choice only when ${\rm Sh}(\K^{\rm op},J_{\rm at})\simeq{\bf Set}^I$.
Thus, in a sense, it holds only in the trivial cases.

\subsection{AC/DC criterion for finite structures}

Since Theorem \ref{thm:main-stability-thm} concerns classes of infinite structures, it does not apply to classes of finite structures.
However, most of the prior research mentioned in Section \ref{sec:Introduction} deals with classes of finite structures.
Therefore, it is also important to establish criteria for the validity of the axiom of choice for atomic toposes over classes of finite structures.


As seen in Section \ref{exa:IAC}, the internal axiom of choice holds when a specific finite upper bound is imposed.
In fact, the validity of the axiom of choice for the atomic topos over a class $\K$ of finite structures depends on whether there is a finite upper bound on the cardinalities of the structures in $\K$.
More precisely, the essential factor is the existence of such a finite upper bound for each connected component.

Here, a category can be viewed as an undirected graph if the direction of morphisms is disregarded.
The connected components of a category are the connected components of such an undirected graph.

\begin{theorem}\label{thm:main-finite-structures}
Let $\K$ be a subcanonical AP class consisting of finite structures; that is, $|A|$ is a finite set for every object $A\in\K$.
Then the following are equivalent:
\begin{enumerate}
\item For every connected component $\K'$ of $\K$, there exists a number $b\in\N$ such that every object $A\in \K'$ has cardinality at most $b$.
\item ${\rm Sh}(\K^{\rm op},J_{\rm at})$ satisfies the internal axiom of countable choice.
\item ${\rm Sh}(\K^{\rm op},J_{\rm at})$ satisfies the internal axiom of dependent choice.
\item ${\rm Sh}(\K^{\rm op},J_{\rm at})$ satisfies the internal axiom of choice.
\end{enumerate}
\end{theorem}

\begin{proof}
(4)$\Rightarrow$(3)$\Rightarrow$(2):
Trivial.

(1)$\Rightarrow$(4):
By Theorem \ref{thm:IAC-maximalcofinal} (1)$\Rightarrow$(2), it suffices to show that $\K$ has the maximal extension property.

For each $A\in\K$, let $[A]_\K$ be the connected component to which $A$ belongs.
By the assumption (1), there exists a structure $M\in [A]_\K$ which has a maximal cardinality within $[A]_\K$.
To see that this $M$ is a maximal structure, let $M\hookrightarrow B$ be a morphism.
As $B$ and $M$ belong to the same connected component, the cardinality of $M$ is greater than or equal to that of $B$.
In this case, the injection $M\hookrightarrow B$ must be a bijection; thus, since $|-|$ is conservative, $M\hookrightarrow B$ is an isomorphism.
Hence, $M$ is a maximal structure.

It is clear that if $\K$ has the AP, then each connected component $[A]_\K$ also has the AP.
Thus, by \cite[Lemma 3.7]{Car14}, $[A]_K$ has the JEP, so there exists a morphism $A\embed B\leftembed M$.
Thus, the maximality of $M$ implies $B\simeq M$, yielding a morphism $A\embed M$.

(2)$\Rightarrow$(1):
Suppose there exists a connected component $\K'$ that does not satisfy condition (1).
Fix $A\in\K'$.
Then, by unboundedness, for any $n\in\N$, there exists some $B_n\in\K'$ whose cardinality  is at least $n$.
Since $\K$ has the AP, again by \cite[Lemma 3.7]{Car14}, the connected component $\K'$ has the JEP; hence, we obtain a morphism $A\embed C_n\leftembed B_n$.
Now, suppose that $(A\embed C_n)_{n\in\N}$ admits a cocone $(C_n\embed D)_{n\in\N}$ in $\K$.
Then, this yields an injection $B_n\embed C_n\embed D$, so the cardinality of $D$ is at least that of $B_n$; that is, at least $n$.
Since this holds for any $n\in\N$, the cardinality of $D\in\K$ must be infinite, which is impossible.
By Theorem \ref{lem:cone-characterization-countable-choice}, this shows the failure of the internal axiom of countable choice in the corresponding atomic sheaf topos.
\end{proof}

For most natural classes $\K$ of structures, the restriction $\K_{\rm fin}$ to finite structures does not satisfy the condition (1) in Theorem \ref{thm:main-finite-structures}; thus, ${\rm Sh}(\K_{\rm fin}^{\rm op},J_{\rm at})$ does not even satisfy the axiom of countable choice.
For instance, the Schanuel topos ${\rm Sh}({\rm Set}_{\rm fin}^{\rm op},J_{\rm at})$ does not satisfy the axiom of countable choice \cite{SGL,Mej19}.

\section{Set-theoretic permutation models}\label{sec:set-theory}

A topos can be converted into a model of a certain weak set theory; see e.g.~\cite{BlSc89}
However it is not clear whether the validity/failure of the axiom of choice is preserved under such a conversion.
Here we compare the atomic sheaf topos ${\rm Sh}(\K^{\rm op},J_{\rm at})$ with the permutation model obtained from the automorphism group ${\rm Aut}(M_\K)$ of a monster structure $M_\K\in\K$; the latter is a model of ZFA, set theory with atoms.
The results we present in this section are summarized in Table \ref{table-2}.

\begin{table}[t]\centering\footnotesize
\begin{tabular}{ccccc}
Atomic topos& Class $\mathcal{K}$ & Continuous  &  Zapletal's & Permutation  \\
${\rm Sh}(\K^{\rm op},J_{\rm at})$ & of structures & ${\rm Aut}(M_\K)$-actions & criteria \cite{Zap26}  & model $W[[M_\K]]$ \\ \hline
${\rm AC}$ for (external)  & cocone & nonemptiness & dynamical &${\rm AC}$ for  \\
$\kappa$-families &  for $\kappa$-fans & of $\kappa$-products &  $\kappa$-completeness & internal $\kappa$-families \\[0.5em]
${\rm DC}$ & cocone & nonemptiness & DC-completeness &${\rm DC}$  \\
&  for $\om$-chains & of inverse limits &   & \\[0.5em]
${\rm AC}$ for & universal & & cofinal &${\rm AC}$ for  \\
(external) sets & extension & & orbit & internal well-orders \\[1em]
\end{tabular}
\caption{\small The correspondence between the choice-criteria for atomic toposes and those for permutation models in set theory.}\label{table-2}
\end{table}

\subsection{Monster structures}

We first introduce an important tool for converting our topos models into set-theoretic models.

\begin{definition}
We call $M\in\K$ a {\bf $\lambda$-monster} structure if it satisfies the following conditions.
\begin{enumerate}
\item $\lambda$-universality:
For every $A\in\K_\lambda$, there exists an embedding $A\embed M$.
\item $\lambda$-ultrahomogeneity:
Let $A,B\in\K_\lambda$ be structures with embeddings $A\embedge{i}M$ and $B\embedge{j}M$.
For every isomorphism $f\colon A\edge{\simeq}B$, there exists an automorphism $\sigma\in{\rm Aut}_\K(M)$ such that $\sigma\circ i=j\circ f$.
\[
\xymatrix{
M\ar[rr]_\sigma^{\simeq} & & M \\
A \ar[u]^i\ar[rr]_f^{\simeq} & & B\ar[u]_j
}
\]
\item The $\lambda$-L\"owenheim--Skolem ($\lambda$-LS) property:
For every subset $S\subseteq|M|$ of cardinality at most $\lambda$, there exist $A\in\K_\lambda$ and an embedding $A\embedge{e}M$ such that $S\subseteq e[A]$.
\end{enumerate}
\end{definition}

\begin{example}
If $T$ is a complete first-order theory, and if an infinite cardinal $\lambda$ is greater than or equal to the cardinality of the language of $T$, then the class $\K={\rm Mod}(T)$ of models of $T$ with elementary embeddings has a $\lambda$-monster structure; see \cite[Definition 6.15 and Theorem 6.16]{TeZi12}.
\end{example}

Again, when morphisms are embeddings (not necessarily elementary), a monster structure may not exist.
A general approach to constructing a monster structure is a form of Fra\"iss\'e limit construction.
However, ultrahomogeneity in Fra\"iss\'e theory concerns extending isomorphisms between finite substructures to automorphisms and is weaker than the ultrahomogeneity required here.
Thus one must instead use the Fra\"iss\'e--J\'onsson construction \cite{Jon60}, which generalizes the Fra\"iss\'e construction to infinite cardinals.

For several classes of structures (with embeddings) mentioned in Section \ref{sec:elementary-example-embedding}, monster structures admit the following concrete constructions.

\begin{example}~
\begin{enumerate}
\item Let $\K={\rm Vec}_F$ be the class of all vector spaces.
If the cardinality of $F$ is at most $\lambda$, then any $F$-vector space satisfying $\dim M>\lambda$ is a $\lambda$-monster.
\item If $\K={\rm LO}$, a $\lambda^+$-saturated/strongly homogeneous DLO (dense linear order without endpoints) provides a $\lambda$-monster structure.
\item If $\K={\rm Graph}$, a sufficiently large saturated/strongly homogeneous random graph provides a $\lambda$-monster structure.
\item If $\K={\rm BA}$, a sufficiently large saturated/strongly homogeneous atomless Boolean algebra provides a $\lambda$-monster.
\end{enumerate}
\end{example}

For (2), (3), and (4), monster structures are constructed by taking saturated and strongly homogeneous models of the model completions of the relevant theories.
For groups and abelian groups, they form J\'onsson classes \cite{Jon60,BHK13}, so at cardinals satisfying the usual J\'onsson cardinal-arithmetic hypotheses, the Fra\"iss\'e--J\'onsson construction yields homogeneous-universal structures.
For other abstract conditions under which monster structures can be constructed, see, for example, \cite{Kub14}.

By transporting a given structure into the monster structure $M$, we may essentially restrict our discussion to substructures of $M$.
This makes comparison with set-theoretic notions considerably easier.

\begin{definition}
For $A,B\in\K$, we write $A\leq B$ if $|A|\subseteq|B|$ and the inclusion map $A\embedge{\subseteq}B$ is a morphism.
In this case, we say that $A$ is a {\bf substructure} of $B$.
\end{definition}

\begin{lemma}\label{lem:substructure-lem}
Suppose that $A\leq M$ and $A\leq B\in\K_\lambda$.
Then there exists an isomorphism $\tau\colon B\edge{\simeq}B'$ fixing every element of $A$ such that $A\leq B'\leq M$.
\end{lemma}

\begin{proof}
More generally, we show that every embedding $A\embedge{i}M$ with $A\in\K_\lambda$ is a universal extension in the following sense:
for every extension $A\embedge{e}B\in\K_\lambda$, there exists $B\embedge{\tau}M$ such that $\tau e=i$.

By universality, fix any embedding $B\embedge{u}M$.
We now have two embeddings $A\embedge{e}B\embedge{u}M$ and $A\embedge{i}M$.
By ultrahomogeneity, the identity map $A\embedge{{\rm id}_A}A$ extends to an automorphism $\sigma\in{\rm Aut}(M)$.
\[
\xymatrix{
	M \ar[rr]^{\simeq}_\sigma & & M\\
	A \ar[u]^{ue} \ar[rr]^{\simeq}_{{\rm id}_A} & & A \ar[u]_{i}
}
\]
Since $\sigma  ue=i$, the morphism $\sigma u\colon B\embed M$ has the desired property.

If $i$ and $e$ are inclusion maps, then $\sigma u(x)=x$.
Thus $\sigma u$ fixes every element of $A$.
By transportability, the bijection $B\edge{\simeq}\sigma u[B]$ gives an isomorphism $B\simeq \sigma u[B]\in\K_\lambda$.
Hence, setting $B'=\sigma u[B]$, we obtain $A\leq B'\leq M$.
\end{proof}

\begin{remark}
In general, even if $A\embedge{e}B\in\K_\lambda$ is not an inclusion map, we still obtain $A\leq B'\leq M$ for some $B'\simeq B$.
\end{remark}

\begin{remark}
By Lemma \ref{lem:substructure-lem}, the $\lambda$-LS property may, for example, be replaced by the following formulation:
for every subset $s\subseteq|M|$ of cardinality at most $\lambda$, there exists a substructure $A\in\K_\lambda$ of $M$ containing $s$.
\end{remark}

\subsection{Concrete formulations of Mejak's criteria}

The choice criteria in Theorems \ref{lem:cone-characterization-countable-choice} and \ref{lem:cone-characterization-dependent-choice}, due to Mejak \cite{Mej19}, show that the validity of choice principles in the atomic sheaf topos ${\rm Sh}(\K^{\rm op}_\lambda,J_{\rm at})$ depends on the existence of certain cocones (generalized amalgamations) in $\K_\lambda$.
\begin{itemize}
\item ${\rm AC}_\kappa$: Every $\kappa$-fan $\{A\embedge{e_\alpha}B_\alpha\}_{\alpha<\kappa}$ in $\K_\lambda$ admits a cocone.
\item ${\rm DC}$: Every $\om$-chain $\{A_n\embedge{e_n}A_{n+1}\}_{n\in\N}$ in $\K_\lambda$ admits a cocone.
\end{itemize}

Independently, Zapletal \cite{Zap26} introduced choice criteria in the context of set-theoretic permutation models.
Remarkably, Mejak's cocone/amalgamation criteria in topos theory and Zapletal's criteria in set theory turn out to be nearly identical.
To bring the topos-theoretic setting closer to the set-theoretic one, we first use transportability to give the following concrete formulation of Mejak's AC criterion.

\begin{lemma}\label{lem:Mejak-AC-criterion}
For a class of structures $\K$, the following are equivalent.
\begin{enumerate}
\item Every $\kappa$-indexed family $\{A\embedge{e_\alpha}B_\alpha\}_{\alpha<\kappa}$ of $\K_\lambda$-morphisms admits a cocone.
\item Suppose that $\kappa$ extensions $\{A\leq B_\alpha\in\K_\lambda\}_{\alpha<\kappa}$ of a structure $A\in\K_\lambda$ are given.
Then there exist isomorphisms $B_\alpha\simeq B_\alpha'\in\K_\lambda$ fixing every element of $A$ such that these $B'_\alpha$ have a common superstructure $\{B_\alpha'\leq C\in\K_\lambda\}_{\alpha<\kappa}$.
\end{enumerate}
\end{lemma}

\begin{proof}
(1)$\Rightarrow$(2):
By the definition of substructure, the family of inclusions $\{A\leq B_\alpha\}_{\alpha<\kappa}$ is a family of $\K_\lambda$-morphisms; hence, by assumption (1), it admits a cocone $\{B_\alpha\embedge{e_\alpha}C\}_{\alpha<\kappa}$.
By the definition of a cocone, $e_\alpha|_A$ is independent of $\alpha$.
Rename the elements of $C$ as follows.
Rename every element of the form $e_\alpha(x)$ as $x$, and rename every other element $y$ as $y'\not\in A$.
Let $C'$ be the result of this renaming. By transportability, the bijection $r\colon C'\edge{\simeq}C$ is an isomorphism, and $C'\in\K_\lambda$.
Setting $B'_\alpha:=r e_\alpha[B_\alpha]$ gives an isomorphism $B_\alpha\simeq B_\alpha'\in\K_\lambda$.
Moreover, $A=re_\alpha[A]\subseteq re_\alpha[B_\alpha]=B_\alpha'\subseteq r[C]=C'$.
This $C'$ is a desired superstructure.

(2)$\Rightarrow$(1):
Suppose that $\{A\embedge{e_\alpha}B_\alpha\}_{\alpha<\kappa}$ is given.
Rename the elements of $B_\alpha$ as follows.
Rename every element of the form $e_\alpha(x)$ as $x$, and rename every other element $y$ as a distinct element $y^\alpha$.
If $B_\alpha'$ denotes the result, then $A\subseteq B_\alpha'$, and by transportability the bijection $\rho_\alpha$ induced by the renaming is an isomorphism $\rho_\alpha\colon B_\alpha\edge{\simeq} B_\alpha'\in\K_\lambda$.
Moreover, by construction, $\rho_\alpha e_\alpha$ is the inclusion map $A\embed B_\alpha'$.

Thus, by assumption (2), we obtain an isomorphism $\sigma_\alpha\colon B_\alpha'\simeq B_\alpha''$ fixing every element of $A$ and an inclusion map $B_\alpha''\leq C\in\K_\lambda$.
Then $\{B_\alpha\embedge{\rho_\alpha}B_\alpha'\embedge{\sigma_\alpha}B_\alpha''\embed C\}_{\alpha<\kappa}$ is a cocone.
Indeed, for every $x\in A$, we have $\sigma_\alpha\rho_\alpha e_\alpha(x)=\sigma_\alpha(x)=x$.
\end{proof}

Similarly, transportability yields the following concrete formulation of Mejak's DC criterion.
\begin{lemma}\label{lem:Mejak-DC-criterion}
For a class of structures $\K$, the following are equivalent.
\begin{enumerate}
\item Every $\om$-chain $A_0\embedge{e_0}A_1\embedge{e_1}A_2\embedge{e_2}\cdots$ of $\K_\lambda$-morphisms admits a cocone.
\item Suppose that a sequence of extensions $A_0\leq A_1\leq A_2\leq\dots$ of $K_\lambda$-structures is given.
Then there exist a sequence of isomorphisms $\{\sigma_n\colon A_{n}\edge{\simeq}A_{n}'\}_{n\in\N}$ and a structure $C\in\K_\lambda$ such that $\sigma_0={\rm id}_{A_0}$ and $\sigma_{n+1}|_{A_n}=\sigma_n$, and
\[A_0=A_0'\leq A_1'\leq \cdots\leq A_n'\leq A_{n+1}'\leq \dots\leq C.\]
\end{enumerate}
\end{lemma}

\begin{proof}
(1)$\Rightarrow$(2):
By assumption, the chain of inclusion maps $\{A_n\leq A_{n+1}\}_{n\in\N}$ admits a cocone $\{A_n\embedge{e_n}C\}_{n\in\N}$.
By the definition of a cocone, $e_{n+1}|_{A_n}=e_n$.
Rename the elements of $C$ as follows.
Rename every element of the form $e_0(x)$ as $x$, and rename every other element $y\in C$ as a new element $y'$.
Let $C'$ be the result. By transportability, the bijection $r\colon C'\edge{\simeq}C$ is an isomorphism, and $C'\in\K_\lambda$.
Setting $A'_n:=r e_n[A_n]$ likewise gives an isomorphism $A_n\simeq A_n'\in\K_\lambda$, and $re_0[A_0]=A_0$.
Moreover, $e_{n+1}|_{A_n}=e_n$ implies $re_{n+1}|_{A_n}=re_n$.
Since $A_n\subseteq A_{n+1}$ by assumption, we have $A_n'=re_n[A_n]\subseteq re_{n+1}[A_n]\subseteq re_{n+1}[A_{n+1}]=A_{n+1}'\subseteq r[C]=C'$.

(2)$\Rightarrow$(1):
Suppose that a chain $\{A_n\embedge{e_n}A_{n+1}\}_{n\in\N}$ is given.
Inductively, suppose that $\rho_n\colon A_n\edge{\simeq}A_n'$ has been constructed.
Rename the elements of $A_{n+1}$ as follows.
Rename every element of the form $e_n(x)$ as $\rho_n(x)$, and rename every other element $y$ as a distinct element $y^n$.
If $A_{n+1}'$ denotes the result, then by transportability the bijection $\rho_{n+1}$ induced by the renaming is an isomorphism $\rho_{n+1}\colon A_{n+1}\edge{\simeq} A_{n+1}'\in\K_\lambda$.
By definition, $\rho_{n+1}e_n(x)=\rho_n(x)$.
In particular, $A_n'=\rho_n[A_n]=\rho_{n+1}e_n[A_n]\subseteq\rho_{n+1}[A_{n+1}]=A_{n+1}'$.

Applying assumption (2) to the sequence $\{A_n'\leq A_{n+1}'\}_{n\in\N}$, we obtain isomorphisms $\sigma_n\colon A_n'\simeq A_n''$ and inclusion maps $A_n''\leq C\in\K_\lambda$ such that $\sigma_{n+1}|_{A_n'}=\sigma_n$.
Then $\{A_n\embedge{\rho_n}A_n'\embedge{\sigma_n}A_n''\embed C\}_{n\in\N}$ is a cocone.
Indeed, for every $x\in A_n$, we have $\sigma_{n+1}\rho_{n+1} e_{n}(x)=\sigma_{n+1}\rho_n(x)=\sigma_{n}\rho_n(x)$.
\end{proof}

\subsection{Dynamical completeness}

We next introduce Zapletal's choice criteria \cite{Zap26}.
Zapletal's framework applies to general group actions, but restricting it to automorphism groups brings it very close to our setting.
Let $\mathcal{P}_\lambda X$ denote the family of all subsets $S$ of a set $X$ such that the cardinality of $S$ is at most $\lambda$.
Given a structure $M\in\K$, the family $\mathcal{P}_\lambda M$ is a dynamical ideal (in the sense of \cite{Zap26}) for the obvious action ${\rm Aut}_\K(M)\acts M$ of the automorphism group.

For a structure $M\in\K$ and a subset $a\subseteq M$, we denote the set of all automorphisms of $M$ fixing every element of $a$ by ${\rm Aut}_\K(M/a)$, or simply by ${\rm Aut}(M/a)$.
\[
{\rm Aut}_\K(M/a)=\{\sigma\in{\rm Aut}_\K(M):\forall x\in a.\ \sigma(x)=x\}.
\]

For automorphism groups, Zapletal's AC criterion can be reformulated as follows.

\begin{definition}[Zapletal {\cite[Definition 5.2]{Zap26}}]\label{def:Zap-dynamical-completeness}
$(M,\lambda)$ is {\bf dynamically $\kappa$-complete} if, for every $a\in \mathcal{P}_\lambda M$ and every $\{ a \subseteq b_\alpha\in \mathcal{P}_\lambda M\}_{\alpha<\kappa}$, there exists $\{\sigma_\alpha\in{\rm Aut}_\K(M/a)\}_{\alpha<\kappa}$ such that $\bigcup_{\alpha<\kappa}\sigma_\alpha[b_\alpha]\in \mathcal{P}_\lambda M$.
\end{definition}

In this form, its close resemblance to Mejak's concrete AC criterion (Lemma \ref{lem:Mejak-AC-criterion} (2)) is apparent.
Indeed, for sufficiently well-behaved classes of structures, Mejak's and Zapletal's AC criteria are equivalent.

\begin{theorem}\label{thm:dynamical-completeness}
Let $\lambda$ be an infinite cardinal, and suppose that a structure-class $\K$ is coherent and has a $\lambda$-monster structure $M\in\K$.
For every infinite cardinal $\kappa$, the following are equivalent.
\begin{enumerate}
\item Every $\kappa$-indexed family $\{A\embedge{e_\alpha}B_\alpha\}_{\alpha<\kappa}$ of $\K_\lambda$-morphisms admits a cocone.
\item $(M,\lambda)$ is dynamically $\kappa$-complete.
\end{enumerate}
\end{theorem}

\begin{proof}
(1)$\Rightarrow$(2):
Take $a\in \mathcal{P}_\lambda M$ and $(b_\alpha\in \mathcal{P}_\lambda M)_{\alpha<\kappa}$.
First, the LS property gives a substructure $A\in\K_\lambda$ of $M$ containing $a$.
Applying the LS property again, for each $\alpha<\kappa$ we obtain a substructure $B_\alpha\in\K_\lambda$ of $M$ containing $A\cup b_\alpha$.
By coherence, the inclusion $A\subseteq B_\alpha$ gives a substructure relation $A\leq B_\alpha$.

Thus, by assumption (1) and Lemma \ref{lem:Mejak-AC-criterion}, there exist isomorphisms $\sigma_\alpha\colon B_\alpha\simeq B_\alpha'\in\K_\lambda$ fixing every element of $A$ and a common superstructure $\{B_\alpha'\leq C\in\K_\lambda\}_{\alpha<\kappa}$.
By Lemma \ref{lem:substructure-lem}, we get an isomorphism $\tau\colon C\edge{\simeq}C'\in\K_\lambda$ preserving $A$ such that $C'\leq M$.
The composite $\tau\sigma_\alpha$ gives an isomorphism $B_\alpha\edge{\simeq}\tau[B_\alpha']$ preserving $A$.
By ultrahomogeneity, it extends to an automorphism $\sigma_\alpha'\in{\rm Aut}(M/A)\subseteq{\rm Aut}(M/a)$ preserving $A$.
Then
\[\sigma_\alpha'[b_\alpha]\subseteq\sigma_\alpha'[B_\alpha]=\tau\sigma_\alpha[B_\alpha]=\tau[B_\alpha']\subseteq\tau[C]=C'\]
holds for every $\alpha$.
Therefore, $\bigcup_{\alpha<\kappa}\sigma_\alpha'[b_\alpha]\subseteq C'\in\K_\lambda$, so its cardinality is at most $\lambda$.
Hence $(M,\lambda)$ is dynamically $\kappa$-complete.

(2)$\Rightarrow$(1):
It suffices to show Lemma \ref{lem:Mejak-AC-criterion} (2).
Suppose that extensions $\{A\leq B_\alpha\}_{\alpha<\kappa}$ of $\K_\lambda$-structures are given.
By Lemma \ref{lem:substructure-lem}, there exists an isomorphism $\tau_\alpha$ fixing every element of $A$ such that $A\leq B_\alpha\simeq\tau_\alpha[B_\alpha]\leq M$.

By dynamical completeness, there exist $\sigma_\alpha\in{\rm Aut}_\K(M/A)$ such that $S:=\bigcup_{\alpha<\kappa}\sigma_\alpha\tau_\alpha[B_\alpha]$ has cardinality at most $\lambda$.
The LS property gives a substructure $C\in\K_\lambda$ of $M$ containing $S$.
Since $\sigma_\alpha$ and $\tau_\alpha$ fix every element of $A$, the family $\{B_\alpha\xedge{\sigma_\alpha\tau_\alpha}\sigma_\alpha\tau_\alpha[B_\alpha]\embed C\}_{\alpha<\kappa}$ is a cocone.
As in the proof of Lemma \ref{lem:Mejak-AC-criterion} (1)$\Rightarrow$(2), we then obtain $B_\alpha',C'$ such that $A\leq B_\alpha'\leq C'$.
\end{proof}

For automorphism groups, Zapletal's DC criterion (see also \cite{KaSc24} as a precursor) can be reformulated as follows.

\begin{definition}[Zapletal {\cite[Definition 4.2]{Zap26}}]\label{def:Zap-DC-completeness}
$(M,\lambda)$ is {\bf DC-complete} if, for every sequence $a_0,a_1,a_2,\dots\in\mathcal{P}_\lambda M$, there exists a sequence of automorphisms $\sigma_0,\sigma_1,\sigma_2,\dots\in{\rm Aut}(M)$ such that $\sigma_0\in{\rm Aut}(M/a_0)$, $\sigma_{n+1}|_{a_n}=\sigma_n|_{a_n}$, and $\bigcup_{n\in\N}\sigma_n[a_n]\in\mathcal{P}_\lambda M$.
\end{definition}

In this form, its close resemblance to Mejak's concrete DC criterion (Lemma \ref{lem:Mejak-DC-criterion} (2)) is apparent.
Indeed, for sufficiently well-behaved classes of structures, Mejak's and Zapletal's DC criteria are equivalent.

\begin{theorem}\label{thm:DC-completeness}
Let $\lambda$ be an infinite cardinal, and suppose that a structure-class $\K$ is coherent and has a $\lambda$-monster structure $M\in\K$.
The following are equivalent.
\begin{enumerate}
\item Every chain $\{A_n\embedge{e_n}A_{n+1}\}_{n\in\N}$ of $\K_\lambda$-morphisms admits a cocone.
\item $(M,\lambda)$ is DC-complete.
\end{enumerate}
\end{theorem}

\begin{proof}
(1)$\Rightarrow$(2):
Take any sequence $a_0,a_1,a_2,\dots\in\mathcal{P}_\lambda M$.
Since the cardinality of $a_0$ is at most $\lambda$, the LS property gives a substructure $A_0\in\K_\lambda$ of $M$ containing $a_0$.
Inductively, suppose that $A_n\in\K_\lambda$ has been constructed.
Since the cardinality of $A_n\cup a_{n+1}$ is at most $\lambda$, the LS property gives a substructure $A_{n+1}\in\K_\lambda$ of $M$ containing $A_n\cup a_{n+1}$.
Thus we obtain a sequence $A_0\leq A_1\leq A_2\leq\cdots\leq M$.

By Lemma \ref{lem:Mejak-DC-criterion}, there exist isomorphisms $\sigma_n\colon A_{n}\edge{\simeq}A_{n}'$ satisfying $\sigma_0=1_{A_0}$ and $\sigma_{n+1}|_{A_n}=\sigma_{n}$, together with an increasing chain $A_0=A_0'\leq A_1'\leq \cdots\leq A_n'\leq A_{n+1}'\leq \dots\leq C$.
By Lemma \ref{lem:substructure-lem}, there exists an isomorphism $\tau\colon C\edge{\simeq}C'\leq M$ fixing every element of $A_0$.
Setting $B_n=\tau[A_n']$, we then have $A_0=B_0\leq B_1\leq\cdots\leq C'\leq M$.

By ultrahomogeneity, the isomorphism $\tau\sigma_n\colon A_{n}\edge{\simeq}B_{n}$ extends to an automorphism $\rho_n\in{\rm Aut}_\K(M)$.
Since $\tau\sigma_0$ preserves $A_0$, we have $\rho_0\in{\rm Aut}_\K(M/A_0)\subseteq{\rm Aut}_\K(M/a_0)$.
Moreover, because $\sigma_{n+1}|_{A_n}=\sigma_{n}$, for every $x\in a_n\subseteq A_n$ we have $\rho_{n+1}(x)=\tau\sigma_{n+1}(x)=\tau\sigma_n(x)=\rho_n(x)$.
Hence $\rho_{n+1}|_{a_n}=\rho_n|_{a_n}$.
Also, $\rho_n[a_n]\subseteq\rho_n[A_n]=\tau\sigma_n[A_n]=B_n\subseteq C'$ holds for every $n$.
Thus $\bigcup_{n\in\N}\rho_n[a_n]\subseteq C'\in\K_\lambda$, so its cardinality is at most $\lambda$.

(2)$\Rightarrow$(1):
It suffices to show Lemma \ref{lem:Mejak-DC-criterion} (2).
Suppose that a chain of inclusion maps $A_0\leq A_1\leq A_2\leq\cdots$ is given.
Inductively construct isomorphisms $\tau_n\colon A_n\edge{\simeq} A_n'\leq M$ such that $\tau_{n+1}|_{A_{n}}=\tau_n$.
By universality, we have an embedding $A_0\embedge{\tau_0} M$, which yields $A_0'=\tau_0[A_0]\simeq A_0$.
Given $A_n\embedge{\tau_n}\tau_n[A_n]\leq M$, since it is a universal extension in the sense of Lemma \ref{lem:substructure-lem}, so we have $\tau_{n+1}\colon A_{n+1}\embed M$ such that $\tau_{n+1}|_{A_n}=\tau_n$.
Identify $\tau_{n+1}$ with the isomorphism $A_{n+1}\edge{\simeq} A_{n+1}':=\tau_{n+1}[A_{n+1}]$.

By DC-completeness, there exist automorphisms $\sigma_n\in{\rm Aut}(M)$ such that $\sigma_n\in{\rm Aut}(M/A_0')$, $\sigma_{n+1}|_{A_n'}=\sigma_n|_{A_n'}$, and $\bigcup_{n\in\N}\sigma_n[A_n']\in\mathcal{P}_\lambda M$.
By transportability, $A_n'\simeq \sigma_n[A_n']\in\K_\lambda$.
Furthermore, since $\sigma_n[A_n']=\sigma_{n+1}[A_n']\subseteq\sigma_{n+1}[A_{n+1}']$, we have $\sigma_n[A_n']\leq \sigma_{n+1}[A_{n+1}']$.
Since each $A_n'':=\sigma_n[A_n']$ has cardinality at most $\lambda$, the cardinality of $\bigcup_{n\in\N} A_n''$ is also at most $\lambda$.
The LS property therefore gives a substructure $C\in\K_\lambda$ of $M$ such that $\bigcup_{n\in\N}A_n''\subseteq C$.
Thus we obtain a sequence $A_0=A_0'\leq A_1''\leq A_2''\leq\cdots\leq C$.
Then $\{A_n\embedge{\tau_n}A_n'\embedge{\sigma_n}A_n''\embed C\}$ is a cocone.
The remainder follows by the same argument as in Lemma \ref{lem:Mejak-DC-criterion} (1)$\Rightarrow$(2).
\end{proof}

\subsection{Cofinal orbits}

Zapletal \cite{Zap26} also gave a criterion for well-ordered choice.
It is formulated in terms of cofinality of orbits.

\begin{definition}
For $S\subseteq M$, define its orbit over $X$ by
\[
{\rm Orb}_M(S;X)=\{\sigma[S]:\sigma\in {\rm Aut}(M/X)\}
\]

For $x\in M$, we use ${\rm Orb}_M(x;X)$ as an abbreviation for ${\rm Orb}_M(\{x\};X)$.
\end{definition}

\begin{definition}[Zapletal {\cite[Definition 3.2]{Zap26}}]
$(M,\lambda)$ has a {\bf cofinal orbit} if, for every $a\in \mathcal{P}_\lambda M$, there exists $b\in \mathcal{P}_\lambda M$ such that ${\rm Orb}_M(b;a)$ is cofinal in $\mathcal{P}_\lambda M$:
that is, for every $c\in \mathcal{P}_\lambda M$, there exists $\sigma\in {\rm Aut}(M/a)$ such that $c\subseteq\sigma[b]$.
\end{definition}

Remarkably, this is connected with the AEC-theoretic universal extension property (Definition \ref{def:universal-extension-prop}).

\begin{theorem}\label{thm:cofinal-orbit}
Let $\lambda$ be an infinite cardinal, and suppose that a structure-class $\K$ is coherent and has a $\lambda$-monster structure $M\in\K$.
Then the following are equivalent.
\begin{enumerate}
\item $\K_\lambda$ has the universal extension property.
\item $\mathcal{P}_\lambda M$ has a cofinal orbit.
\end{enumerate}
\end{theorem}

\begin{proof}
(1)$\Rightarrow$(2):
Take any $a\in \mathcal{P}_\lambda M$.
By the LS property, there exists a substructure $A\in\K_\lambda$ of $M$ containing $a$.
Choose a universal extension $A\embed U_A\in\K_\lambda$ over $A$.
By Lemma \ref{lem:substructure-lem}, we may assume that $A\leq U_A\leq M$.
In particular, $U_A\in\K_\lambda$.
We show that the orbit of $U_A$ over $a$ is cofinal.

For any $c\in \mathcal{P}_\lambda M$, the LS property gives a substructure $C\leq M$ containing $A\cup c$.
Since $U_A$ is a universal extension over $A$, there exists $C\embedge{e}U_A$ fixing every element of $A$.
By ultrahomogeneity, the bijection $C\edge{\simeq}e[C]$ extends to an automorphism $\hat{e}\in{\rm Aut}(M)$.
In particular, $\hat{e}$ fixes $a\subseteq A$, so $\hat{e}\in {\rm Aut}(M/a)$; moreover, $\hat{e}[c]\subseteq \hat{e}[C]=e[C]\subseteq U_A$.
Hence $c\subseteq\hat{e}^{-1}[U_A]$.
Since $\hat{e}^{-1}\in{\rm Aut}(M/a)$, the orbit of $U_A$ over $a$ is cofinal.

(2)$\Rightarrow$(1):
Take $A\in\K_\lambda$.
By universality and transportability, we may assume that $A\leq M$.
By assumption, the orbit of some $b\in \mathcal{P}_\lambda M$ over $A$ is cofinal.
The LS property gives a substructure $B\leq M$ containing $A\cup b$.
By coherence, the inclusion $A\subseteq B$ then gives an embedding $A\embedge{u}B$.
We show that $A\embedge{u}B$ is universal.

For any embedding $A\embedge{i} N\in\K_\lambda$, Lemma \ref{lem:substructure-lem} allows us to assume that $A\leq N\leq M$.
Thus $i$ is an inclusion map.
Since the orbit of $b$ over $A$ is cofinal, there exists $\sigma\in {\rm Aut}(M/A)$ such that $N\subseteq \sigma[b]$.
It follows that $\sigma^{-1}[N]\subseteq b\subseteq B$.
Since $\sigma$ fixes $A$, we have $A\subseteq \sigma^{-1}[N]$, and the isomorphism $N\edge{\simeq}\sigma^{-1}[N]\subseteq B$ fixes $A$.
By coherence, this again gives an embedding $N\embedge{j} B$.
Since $u,i,j$ are inclusion maps, $ji(x)=x=u(x)$ for every $x\in A$.
\end{proof}

\subsection{Definable closure}

Let $\K$ be a structure-class with a $\lambda$-monster structure $M\in\K$.

\begin{definition}
An {\bf epi-span} of a subset $A\subset M$ is a substructure $X\in\K_\lambda$ of $M$ containing $A$ such that, for all $f,g\colon X\to Y\in\K_\lambda$, the equality $f|_A=g|_A$ implies $f=g$.
We denote it by $\pair{A}$.
\end{definition}

\begin{definition}
We say that $\K$ admits {\bf one-point spans} if, for every substructure $A\in\K_\lambda$ of $M$ and every $x\in M$, there exists an epi-span $\pair{A\cup\{x\}}\in\K_\lambda$ of $A\cup\{x\}$.
\end{definition}

This is a weaker version of pseudo-universality \cite{Vas17} in the context of AEC.
Therefore, a pseudo-universal AEC always admits one-point span.

Now, we consider the following model-theoretic property.

\begin{definition}
The {\bf definable closure} of a subset $S\subseteq M$ is defined as follows:
\begin{align*}
{\rm dcl}_M(S)&=\{x\in M:\#{\rm Orb}_M(x;S)=1\}\\
&=\{x\in M:\forall \sigma\in{\rm Aut}(M/S),\sigma(x)=x\}
\end{align*}
\end{definition}

In the context of Fra\"{i}ss\'e theory, the strong AP is known to be equivalent to the triviality of algebraic/definable closure; see e.g.~\cite[Section 7.1]{Hod93}.
For a recent work, see also \cite[Corollary 5.11]{YoZa26}.
In a general setting that deals with structures beyond the finite substructures we consider, this equivalence breaks down significantly.
However, even in our abstract setting, we show that the purely topos-theoretic condition of subcanonicity can, remarkably, characterize the triviality of definable closure.

\begin{theorem}\label{thm:subcanonical-definably-closed}
Let $\lambda$ be an infinite cardinal, and suppose that a coherent structure-class $\K_\lambda$ has the AP, one-point spans, and a monster structure $M\in\K$.
Then the following are equivalent.
\begin{enumerate}
\item The atomic topology $J_{\rm at}$ on $\K_\lambda^{\rm op}$ is subcanonical.
\item For every substructure $A\in\K_\lambda$ of $M$, ${\rm dcl}_M(A)=A$.
\end{enumerate}
\end{theorem}

\begin{proof}
(1)$\Rightarrow$(2):
Take a substructure $A\in\K_\lambda$ of $M$.
For $x\in M\setminus A$, we wish to construct an automorphism $\sigma\in{\rm Aut}(M/A)$ that moves $x$.
The epi-span condition gives $B=\pair{A\cup\{x\}}\in\K_\lambda$.
By subcanonicity, the inclusion map $A\embedge{i} B$ is a strict monomorphism (Lemma \ref{lem:subcanonical-strict}).
On the other hand, since $x\not=i(a)$ for every $a\in A$, there is no $B\embedge{a}A$ such that $1_B=ia$.
Thus, by the contrapositive of strictness, $1_B$ is not an $i$-matching element; that is, there are embeddings $f,g\colon B\embed C$ such that $f|_A=g|_A$ but $f\not=g$.
By the definition of an epi-span, we must have $f(x)\not=g(x)$.
By universality, there is an embedding $C\embedge{u} M$; the morphisms $uf,ug\colon B\to M$ satisfy $uf|_A=ug|_A$ and, by injectivity, $uf(x)\not=ug(x)$.
By ultrahomogeneity, take an extension $\sigma\in{\rm Aut}(M)$ of the isomorphism $B\simeq uf[B]$, and consider $\sigma^{-1}ug\colon B\to M$.
For any $a\in A$, we have $\sigma^{-1}ug(a)=\sigma^{-1}uf(a)=a$.
On the other hand, $\sigma(x)=uf(x)\not=ug(x)$, so $\sigma^{-1}ug(x)\not=x$.
Using ultrahomogeneity again, extend the isomorphism $B\simeq\sigma^{-1} ug[B]$ to an automorphism $\tau\in{\rm Aut}(M)$.
Then $\tau\in{\rm Aut}(M/A)$ and $\tau(x)\not=x$, so $x\not\in{\rm dcl}(A)$.
Therefore, ${\rm dcl}(A)=A$.

(2)$\Rightarrow$(1):
By Lemma \ref{lem:subcanonical-strict}, it suffices to show that every embedding $A\embedge{i}B$ is a strict monomorphism.
Take $X\embedge{x}B$ and suppose that no $\hat{x}$ satisfies $x=i\hat{x}$.
It suffices to show that $x$ is not an $i$-matching element.

By universality, choose $B\embedge{u}M$.
If $x[X]\subseteq i[A]$, coherence gives $i^{-1}x\colon X\embed A$, so $i\hat{x}=x$, contrary to the assumption.
Thus $x[X]\not\subseteq i[A]$, and hence $x(z)\not\in i[A]$ for some $z\in X$.
Since $u$ is injective, $ux(z)\not\in ui[A]$.
As $A\simeq ui[A]\subseteq M$, transportability allows us to regard $A':=ui[A]\leq M$ as a substructure.
By assumption, ${\rm dcl}(A')=A'$, so $ux(z)\not\in{\rm dcl}(A')$.
Consequently, there exists an automorphism $\sigma\in{\rm Aut}(M/\A')$ such that $\sigma ux(z)\not=ux(z)$.
The cardinality of $u[B]\cup\sigma u[B]$ is at most $\lambda$, so the LS property gives a substructure $C\in\K_\lambda$ of $M$ containing it.
By coherence, restricting the codomains of $u,\sigma u\colon B\to M$ to $C$ again gives embeddings.
Since $\sigma$ fixes every element of $A'$, we have $u i=\sigma ui$.
On the other hand, $u(x(z))\not=\sigma u(x(z))$, and hence in particular $ux\not=\sigma ux$.
Therefore, $x$ is not an $i$-matching element.
\end{proof}

One-point spans are necessary only to show (1)$\Rightarrow$(2).
If $\K$ has the strong amalgamation property, epi-spans are unnecessary.

\begin{prop}\label{prop:strongAP-definably-closed}
Let $\lambda$ be an infinite cardinal, and suppose that the coherent class of structures $\K_\lambda$ has the strong amalgamation property and has a monster structure $M\in\K$.
Then ${\rm dcl}_M(A)=A$ for every substructure $A\in\K_\lambda$ of $M$.
\end{prop}

\begin{proof}
The strong amalgamation property makes it possible to separate completely the part outside a substructure $A\leq B$. Thus, in the presence of a monster structure, one can always construct an automorphism that fixes $A$ but moves everything outside it.

For $x\in M\setminus A$, we wish to construct an automorphism $\sigma\in{\rm Aut}(M/A)$ that moves $x$.
By the LS property, there exists a substructure $B\in\K_\lambda$ of $M$ such that $A\cup\{x\}\subseteq B$.
A strong amalgamation $e_0,e_1\colon B\embed C$ of two copies of the inclusion map $A\embed B$ satisfies $e_0(a)=e_1(a)$ for every $a\in A$ and $e_0(b)\not=e_1(b)$ for every $b\in B\setminus A$.

More precisely, as in the proof of Lemma \ref{lem:substructure-lem}, there first exists $C\embedge{\tau} M$ such that $B\embedge{\tau e_0}M$ fixes $B$.
By ultrahomogeneity, the isomorphism $B\simeq\tau e_1[B]$ extends to an automorphism $\sigma\in{\rm Aut}(M)$.
Then $\sigma(a)=\tau e_1(a)=\tau e_0(a)=a$ for every $a\in A$, while $\sigma(b)=\tau e_1(b)\not=\tau e_0(b)=b$ for every $b\in B\setminus A$.
In particular, $\sigma(x)\not=x$.
Thus $\sigma\in{\rm Aut}(M/A)$ does not fix $x$, and hence $x\not\in{\rm dcl}_M(A)$.
\end{proof}

As a hypothesis for his characterizations of choice principles, Zapletal introduced the definable closedness of an ideal.
For automorphism groups, this notion is expressed as follows.

\begin{definition}
$(M,\lambda)$ is {\bf definably closed} if $a\in\mathcal{P}_\lambda M$ implies ${\rm dcl}(a)\in\mathcal{P}_\lambda M$.
That is, if a subset $a\subseteq M$ has cardinality at most $\lambda$, then ${\rm dcl}(a)$ also has cardinality at most $\lambda$.
\end{definition}

\begin{cor}
Suppose that $\K_\lambda$ is a coherent class of structures with the amalgamation property and has a monster structure $M\in\K$.
Assume, in addition, that one of the following holds.
\begin{enumerate}
\item $\K_\lambda$ admits one-point spans, and $(\K_\lambda^{\rm op},J_{\rm at})$ is subcanonical.
\item $\K_\lambda$ has the strong amalgamation property.
\end{enumerate}

%
%
Then $(M,\lambda)$ is definably closed.
\end{cor}

\begin{proof}
Take a subset $a\subseteq M$ of cardinality at most $\lambda$.
By the LS property, there exists a substructure $A\in\K_\lambda$ of $M$ containing $a$.
For (1), Theorem \ref{thm:subcanonical-definably-closed} gives the desired conclusion; for (2), Proposition \ref{prop:strongAP-definably-closed} gives the same conclusion. Thus, in either case, ${\rm dcl}_M(a)\subseteq{\rm dcl}_M(A)=A$.
Therefore, ${\rm dcl}_M(a)$ also has cardinality at most $\lambda$.
\end{proof}

\subsection{Topological Galois representation}

The correspondences established in Theorems \ref{thm:dynamical-completeness} and \ref{thm:DC-completeness} admit a common interpretation through the topological Galois representation of the atomic topos \cite{Car16,CaLa19}.
In this subsection we briefly explain this interpretation.
Its purpose is to explain conceptually why the Mejak and Zapletal criteria coincide.

%
First, we topologize ${\rm Aut}_\K(M)$ as follows.

\begin{definition}
Let $M\in\K$ be a $\lambda$-monster structure.
Then, ${\rm Aut}(M)$ is equipped with the topology in which ${\rm Aut}(M/A)$ is an open neighborhood of the identity $\text{id}_M$ for each substructure $A\in\K_\lambda$ of $M$.
\end{definition}

By the $\lambda$-LS property, this topology matches the topology generated by $\{{\rm Aut}(M/a):a\in\mathcal{P}_\lambda M\}$.
The notion of definable closedness has the following role in this context.

\begin{lemma}
If $(M,\lambda)$ is definably closed, then for every subset $S\subseteq |M|$,
\[
{\rm Aut}(M/S)\text{ is open}
\iff
S\in P_\lambda M.
\]
\end{lemma}

\begin{proof}
$(\Leftarrow)$
Obvious.
$(\Rightarrow)$
If ${\rm Aut}(M/S)$ is open, there is $a\in P_\lambda M$ such that ${\rm Aut}(M/a)\subseteq{\rm Aut}(M/S)$.
Therefore, every $\sigma\in{\rm Aut}(M/a)$ fixes all elements of $S$; thus, $S\subseteq{\rm dcl}_M(a)$.
Since $(M,\lambda)$ is definably closed, we get $S\subseteq {\rm dcl}_M(a)\in P_\lambda M$.
\end{proof}

For a topological group $G$, a {\bf continuous $G$-set} is a set $X$, equipped with the discrete topology, together with a continuous $G$-action.
Let ${\rm Cont}(G)$ be the topos of continuous $G$-sets; see e.g.~\cite[Section III.9]{SGL}.

\begin{fact}[see e.g.~Caramello {\cite[Theorem 3.5]{Car16}}]\label{fact:caramello}
Assume that $\K_\lambda$ is a subcanonical structure-class which has the AP and a $\lambda$-monster structure $M\in\K$.
\begin{enumerate}
\item ${\rm Sh}(\K_\lambda^{\rm op},J_{\rm at})\simeq{\rm Cont}({\rm Aut}(M))$.
\item For each substructure $A\in\K_\lambda$ of $M$, this equivalence associates $\yo A$ with ${\rm Aut}(M)/{\rm Aut}(M/A)$ 
\item ${\rm Sh}(\K_\lambda^{\rm op},J_{\rm at})/\yo A\simeq{\rm Cont}({\rm Aut}(M/A))$.
\item For each substructures $A\leq B\in\K_\lambda$ of $M$, this equivalence associates $\yo B\to\yo A$ with ${\rm Aut}(M/A)/{\rm Aut}(M/B)$.
\end{enumerate}
\end{fact}

\begin{proof}[Proof (Sketch)]
(1)
By $\lambda$-universality and $\lambda$-LS of $M$, $K_\lambda$ has the JEP.
Indeed, one simply embeds $A,B\in K_\lambda$ into $M$ and includes the union of their two images within a single $K_\lambda$-substructure by the $\lambda$-LS.
Thus, by the AP, the JEP, and the universality and ultrahomogeneity of $M$, one can apply Caramello \cite[Theorem 3.5]{Car16}.
(2)
By $\lambda$-ultrahomogeneity, ${\rm Aut}(M)$ acts transitively on the set ${\rm Emb}(A,M)$ of embeddings from $A$ into $M$.
Since the stabilizer of the inclusion map $A\embed M$ is ${\rm Aut}(M/A)$, we get ${\rm Emb}(A,M)\simeq {\rm Aut}(M)/{\rm Aut}(M/A)$.
(3)
In general, for a topological group $G$ and an open subgroup $U\leq G$, we have ${\rm Cont}(G)/(G/U)\simeq{\rm Cont}(U)$.
(4)
The transitive ${\rm Aut}(M/A)$-set corresponding to $A\leq B$ is ${\rm Aut}(M/A)/{\rm Aut}(M/B)$, by an orbit-stabilizer argument similar to the one above.
\end{proof}

Mejak's choice criterion (Theorems \ref{lem:cone-characterization-countable-choice} and \ref{lem:cone-characterization-dependent-choice}) examines the existence of a cone for a certain diagram sharing a common codomain $A$.
Such a diagram can be understood as a diagram in the slice category $\mathcal{C}/A$.

Let $(\C,J)$ be a site, and fix a $\C$-object $a$.
Given $(b\edge{u} a)\in \C/a$ defines $Y_a(u):=(\yo b\xedge{\yo u}\yo a)\in{\rm Sh}(\C,J)/\yo a$.
This yields a functor $Y_a\colon \C/a\to {\rm Sh}(\C,J)/\yo a$.

\begin{lemma}\label{lemma:cocone-in-slice-topos}
Let $(\C,J_{\rm at})$ be a subcanonical atomic site.
For a $\C$-object $a$, let $d\colon I\to \C/a$ be a diagram, which yields $Y_a(d)\colon I\to{\rm Sh}(\C,J_{\rm at})/\yo a$.
Then the following are equivalent.
\begin{enumerate}
\item $d$ admits a cone in $\C/a$.
\item $\lim_{i\in I} Y_a(d_i)\not\simeq{\bf 0}$.
\end{enumerate}
\end{lemma}

\begin{proof}
(1)$\Rightarrow$(2):
Suppose $d$ admits a cone.
This consists of $(\C/a)$-morphisms $\{c \to d_i\}_{i \in I}$ from a $\mathcal{C}$-morphism $t\edge{c} a$.
Then $\{Y_a(c) \to Y_a(d_i)\}_{i \in I}$ is a cone.
Taking the limit in the sheaf topos, we obtain $Y_a(c)\to\lim_{i\in I}Y_a(d_i)$ by the universal property.
Since $Y_a(c)=(\yo c\to \yo a)$, we have $Y_a(c) \not\simeq \mathbf{0}$.
Thus, we obtain $\lim_{i\in I}Y_a(d_i)\not\simeq\mathbf{0}$.

(2)$\Rightarrow$(1):
Since $J_{\text{at}}$ is subcanonical, $\{\yo c\}_{c\in\C}$ forms a generating family.
Therefore, in the slice category, $\{\yo c\to \yo a\}_{c\in\C}$, that is, $\{Y_a(c)\}_{c\in\C}$, forms a generating family.
Thus, by the assumption $\lim_{i\in I}Y_a(d_i)\not\simeq{\bf 0}$, there exist $t\edge{c}a$ and a morphism $Y_a(c)\to \lim_{i\in I}Y_a(d_i)$.
By composing with the projections, we obtain a cone $\{Y_a(c)\to Y_a(d_i)\}_{i\in I}$.
Since the Yoneda embedding is fully faithful, $\{c\to d_i\}_{i\in I}$ is a cone for $d$.
\end{proof}

\begin{remark}
Fix a structure $A \in K_\lambda$.
For $A\embedge{e}B$, we have $(\yo B\xedge{\yo e}\yo A)\in{\rm Sh(\K_\lambda^{\rm op},J_{\rm at})}/\yo A$, and by Fact \ref{fact:caramello}, this corresponds to ${\rm Aut}(M/A)/{\rm Aut}(M/B)$.
In other words, we obtain $Y_A(A\embed B)\simeq {\rm Aut}(M/A)/{\rm Aut}(M/B)$.
\end{remark}

Thus, Fact \ref{fact:caramello} and Lemma \ref{lemma:cocone-in-slice-topos} provide a translation of Mejak's AC-criterion (Theorem \ref{lem:cone-characterization-countable-choice}) into the language of continuous group actions.
That is, the following are equivalent.
\begin{enumerate}
\item ${\rm Cont}({\rm Aut}(M))$ satisfies the axiom of choice for $\kappa$-families.
\item For every $\kappa$-family $\{A\leq B_\alpha\}_{\alpha<\kappa}$, the product 
\[\prod_{\alpha<\kappa}{\rm Aut}(M/A)/{\rm Aut}(M/B_\alpha)\]
 in ${\rm Cont}({\rm Aut}(M/A))$ is nonempty.
\end{enumerate}

Similarly, Fact \ref{fact:caramello} and Lemma \ref{lemma:cocone-in-slice-topos} provide a translation of Mejak's DC-criterion (Theorem \ref{lem:cone-characterization-dependent-choice}) into the language of continuous group actions.
That is, the following are equivalent.
\begin{enumerate}
\item ${\rm Cont}({\rm Aut}(M))$ satisfies the axiom of dependent choice.
\item For every $\om$-chain $\{A_n\leq A_{n+1}\}_{n\in\om}$, the inverse limit
\[\lim_{n\in\om^{\rm op}}{\rm Aut}(M/A_0)/{\rm Aut}(M/A_n)\]
in ${\rm Cont}({\rm Aut}(M/A_0))$ is nonempty.
\end{enumerate}

Regarding the AC-criterion, we explicitly calculate the product in ${\rm Cont}({\rm Aut}(M))$.
Interestingly, this calculation leads directly to the emergence of Zapletal's dynamical $\kappa$-completeness (Definition \ref{def:Zap-dynamical-completeness}).

\begin{prop}\label{prop:AC-critetion-contG-Zap}
Let $\lambda$ be an infinite cardinal, and suppose that a structure-class $\K$ has a $\lambda$-monster structure $M\in\K$.
Assume that $(M,\lambda)$ is definably closed.
For a $\kappa$-family $(A\leq B_\alpha)_{\alpha<\kappa}$ of extensions, the following are equivalent.
\begin{enumerate}
\item $\displaystyle{\prod_{\alpha<\kappa}{\rm Aut}(M/A)/{\rm Aut}(M/B_\alpha)}$ in ${\rm Cont}({\rm Aut}(M/A))$ is nonempty.
\item There exists $\{\sigma_\alpha\in {\rm Aut}(M/A)\}_{\alpha<\kappa}$ such that $\bigcup_{\alpha<\kappa}\sigma_\alpha[B_\alpha]\in P_\lambda M$.
\end{enumerate}
\end{prop}

To prove this, we first calculate the product explicitly and obtain the following.

\begin{lemma}\label{lem:product-calculation-cont}
The underlying set of the product $\prod_{\alpha<\kappa}{\rm Aut}(M/A)/{\rm Aut}(M/B_\alpha)$ consists of left cosets $(\sigma_\alpha{\rm Aut}(M/B_\alpha))_{\alpha<\kappa}$, where $\sigma_\alpha\in {\rm Aut}(M/A)$ for any $\alpha<\kappa$, such that ${\rm Aut}(M/\bigcup_{\alpha<\kappa}\sigma_\alpha[B_\alpha])$ is open.
\end{lemma}

\begin{proof}
In the following, if a group $H$ acts on $X$, for each $x\in X$, we write ${\rm Stab}_H(x)=\{g\in H:g\cdot x=x\}$.
First, we calculate the product in the category of continuous $H$-actions.

\begin{claim}
Let $H$ be a topological group, and $(X_\alpha)_{\alpha<\kappa}$ be continuous $H$-sets.
Then, the underlying set of their product in ${\rm Cont}(H)$ consists of tuples $(x_\alpha)_{\alpha<\kappa}$, where $x_\alpha\in X_\alpha$ for any $\alpha<\kappa$, such that $\bigcap_{\alpha<\kappa}{\rm Stab}_{H}(x_\alpha)$ is open in $H$.
\end{claim}

\begin{proof}[Proof (Claim)]
An action of $H$ on $X$ is continuous if and only if the stabilizer of each point is open.
Thus, the set in the claim is a continuous $H$-set.
To show the universality, let a continuous $H$-set $Y$ and a morphism $f_\alpha:Y\to X_\alpha$ be given.
Then, for $y\in Y$, we get ${\rm Stab}_{H}(y)\subseteq\bigcap_{\alpha<\kappa}{\rm Stab}_{H}(f_\alpha(y))$.
Since the left-hand side is open, $(f_\alpha(y))_{\alpha<\kappa}$ also has an open stabilizer.
Thus, the induced map $Y\to\prod_{\alpha<\kappa}X_\alpha$ takes values in the subset defined above.
This shows the universal property of the product.
\end{proof}


Now, for each left coset $\sigma_\alpha {\rm Aut}(M/{B_\alpha})\in {\rm Aut}(M/A)/{\rm Aut}(M/{B_\alpha})$, we have ${\rm Stab}_{{\rm Aut}(M/A)}(\sigma_\alpha {\rm Aut}(M/{B_\alpha}))=\sigma_\alpha {\rm Aut}(M/{B_\alpha})\sigma_\alpha^{-1}={\rm Aut}(M/{\sigma_\alpha[B_\alpha]})$.
Moreover, one may see $\bigcap_{\alpha<\kappa}{\rm Aut}(M/{\sigma_\alpha[B_\alpha]})={\rm Aut}(M/{\bigcup_{\alpha<\kappa}\sigma_\alpha[B_\alpha]})$.
By the above claim, this completes the proof.
\end{proof}

\begin{proof}[Proof (Proposition \ref{prop:AC-critetion-contG-Zap})]
(1)$\Rightarrow$(2):
Assume that $\prod_{\alpha<\kappa}{\rm Aut}(M/A)/{\rm Aut}(M/{B_\alpha})$ has an element $(\sigma_\alpha {\rm Aut}(M/{B_\alpha}))_{\alpha<\kappa}$.
Then, by Lemma \ref{lem:product-calculation-cont}, ${\rm Aut}(M/{\bigcup_{\alpha<\kappa}\sigma_\alpha[B_\alpha]})$ is open.
By the definition of our topology, there is $a\in P_\lambda M$ such that ${\rm Aut}(M/a)\subseteq {\rm Aut}(M/{\bigcup_\alpha\sigma_\alpha[B_\alpha]})$.
This implies $\bigcup_\alpha\sigma_\alpha[B_\alpha]\subseteq {\rm dcl}_M(a)$.
By definable closedness, we get $\bigcup_\alpha\sigma_\alpha[B_\alpha]\subseteq{\rm dcl}_M(a)\in P_\lambda M$.

(2)$\Rightarrow$(1):
For $\{\sigma_\alpha \in {\rm Aut}(M/A)\}_{\alpha<\kappa}$ satisfying the condition (2), it is evident from the definition of our topology that $(\sigma_\alpha {\rm Aut}(M/{B_\alpha}))$ satisfies the condition in Lemma \ref{lem:product-calculation-cont}.
\end{proof}

Next, regarding the DC-criterion, we explicitly calculate the inverse limit in ${\rm Cont}({\rm Aut}(M))$.
This calculation leads directly to the emergence of Zapletal's DC-completeness (Definition \ref{def:Zap-DC-completeness}).

\begin{prop}\label{prop:DC-critetion-contG-Zap}
Let $\lambda$ be an infinite cardinal, and suppose that a structure-class $\K$ has a $\lambda$-monster structure $M\in\K$.
Assume that $(M,\lambda)$ is definably closed.
For an $\omega$-chain $(A_n\leq A_{n+1})_{n\in\om}$ of extensions, the following are equivalent.
\begin{enumerate}
\item $\displaystyle{\lim_{n\in\om^{\rm op}}{\rm Aut}(M/A_0)/{\rm Aut}(M/{A_n})}$  in ${\rm Cont}({\rm Aut}(M/A_0))$ is nonempty.
\item There exists $\{\sigma_n\in {\rm Aut}(M/A_0)\}_{n\in\om}$ such that $\sigma_{n+1}|_{A_n}=\sigma_n|_{A_n}$ for any $n\in\om$ and $\bigcup_{n\in\om}\sigma_n[A_n]\in P_\lambda M$.
\end{enumerate}
\end{prop}

To prove this, we calculate the inverse limit explicitly and obtain the following.

\begin{lemma}\label{lem:inverse-limit-calculation-cont}
The underlying set of the inverse limit $\lim_{n\in\om^{\rm op}}{\rm Aut}(M/A_0)/{\rm Aut}(M/A_n)$ consists of left cosets $(\sigma_n{\rm Aut}(M/A_n))_{n\in\om}$, where $\sigma_n\in {\rm Aut}(M/A_0)$ for any $n\in\om$, such that $\sigma_{n+1}|_{A_n}=\sigma_n|_{A_n}$ for any $n\in\om$ and ${\rm Aut}(M/\bigcup_{n\in\om}\sigma_n[A_n])$ is open.
\end{lemma}

\begin{proof}
We calculate the inverse limit in the category of continuous $H$-actions.

\begin{claim}
Let $H$ be a topological group, and $X_0\leftedge{p_0}X_1\leftedge{p_1}X_2\leftedge{p_2}\cdots$ be an inverse system in ${\rm Cont}(H)$.
Then, the underlying set of the inverse limit consists of tuples $(x_n)_{n\in\om}$, where $x_n\in X_n$ for any $n<\om$, such that the following holds.
\begin{enumerate}
\item $p_n(x_{n+1})=x_n$ for every $n\in\om$.
\item $\bigcap_{n\in\om}{\rm Stab}_{H}(x_n)$ is open in $H$.
\end{enumerate}
\end{claim}

The proof of the claim is omitted as it is the same as before.
Now, we apply this claim to $X_n={\rm Aut}(M/{A_0})/{\rm Aut}(M/{A_n})$.
Here, the map 
\[{\rm Aut}(M/{A_0})/{\rm Aut}(M/{A_{n+1}})\edge{p_n} {\rm Aut}(M/{A_0})/{\rm Aut}(M/{A_n})\]
is given by $p_n(\sigma {\rm Aut}(M/{A_{n+1}}))=\sigma {\rm Aut}(M/{A_n})$.
In this case, the condition (1) in the claim applied to $(\sigma_n {\rm Aut}(M/{A_n}))_{n\in\om}$ means $\sigma_{n+1}{\rm Aut}(M/{A_n})=p_n(\sigma_{n+1}{\rm Aut}(M/{A_{n+1}}))=\sigma_n {\rm Aut}(M/{A_n})$.
This implies $\sigma_n^{-1}\sigma_{n+1}\in {\rm Aut}(M/{A_n})$, which means $\sigma_{n+1}|_{A_n}=\sigma_n|_{A_n}$.
The condition (2) is the same as in Lemma \ref{lem:product-calculation-cont}.
\end{proof}

\begin{proof}[Proof (Proposition \ref{prop:DC-critetion-contG-Zap})]
Using Lemma \ref{lem:inverse-limit-calculation-cont}, we obtain the desired assertion by a similar argument as in the proof of Proposition \ref{prop:AC-critetion-contG-Zap}.
\end{proof}

Here, note that in the proofs of Propositions \ref{prop:AC-critetion-contG-Zap} and \ref{prop:DC-critetion-contG-Zap}, the definably closed hypothesis is only used for the direction (1)$\Rightarrow$(2).

In summary, dynamical $\kappa$-completeness is an explicit description of the non-emptiness of products, while DC-completeness is an explicit description of the non-emptiness of inverse limits in ${\rm Cont}({\rm Aut}(M))$.
This yields alternative proofs for Theorems \ref{thm:dynamical-completeness} and \ref{thm:DC-completeness} under a stronger hypothesis (such as the strong AP).

\subsection{Set-theoretic permutation models}
Here, we explain only the idea of constructing a model of set theory from a monster structure; for details, see \cite{Zap26}.
First extend the ${\rm Aut}(M)$-action hereditarily to the cumulative hierarchy, and denote by $W_\lambda[[M]]$ the class of hereditarily $\mathcal{P}_\lambda M$-symmetric sets.
This is a Fraenkel--Mostowski-style permutation model (see, e.g., \cite{JechAC,Zap26}).
Then $W_\lambda[[M]]$ is a model of ZFA.
We now introduce variants of the axiom of choice in the context of set theory.

\begin{definition}[see e.g.~\cite{JechAC,HoRu98}]~
\begin{enumerate}
\item ${\rm AC}_\kappa$ states that every $\kappa$-indexed family of nonempty sets has a choice function.
That is, for every $(X_\xi)_{\xi<\kappa}$ such that $X_\xi\not=\emptyset$, there exists a function $f$ such that $f(\xi)\in X_\xi$.
\item ${\rm AC}_{\rm WO}$ states that every well-ordered family of nonempty sets has a choice function.
That is, ${\rm AC}_\kappa$ holds for every (internal) ordinal $\kappa$.
\item ${\rm DC}$ states that for every total relation $R$ on a nonempty set $X$, there exists a sequence $(x_n)_{n\in\N}$ such that $x_nRx_{n+1}$ for every $n\in\N$.
\end{enumerate}
\end{definition}

We have the following implications in set theory \cite[Section 8.1]{JechAC}:
\[{\rm AC}\implies{\rm AC}_{\rm WO}\implies{\rm DC}\implies{\rm AC}_\om\]
where ${\rm AC}$ is the full axiom of choice.
We now present Zapletal's choice criteria.

\begin{theorem}[Zapletal {\cite[Theorem 5.3]{Zap26}}]\label{thm:Zapletal1}
For a structure $M\in\K$, if $(M,\lambda)$ is definably closed, then the following are equivalent.
\begin{enumerate}
\item $(M,\lambda)$ is dynamically $\kappa$-complete.
\item $W_\lambda[[M]]$ satisfies ${\rm AC}_\kappa$.
\end{enumerate}
\end{theorem}

\begin{theorem}[Zapletal {\cite[Theorem 4.3]{Zap26}}]\label{thm:Zapletal2}
For a structure $M\in\K$, if $(M,\lambda)$ is definably closed, then the following are equivalent.
\begin{enumerate}
\item $(M,\lambda)$ is {\rm DC}-complete.
\item $W_\lambda[[M]]$ satisfies ${\rm DC}$.
\end{enumerate}
\end{theorem}

\begin{theorem}[Zapletal {\cite[Theorem 3.3]{Zap26}}]\label{thm:Zapletal3}
For a structure $M\in\K$, if $(M,\lambda)$ is definably closed, then the following are equivalent.
\begin{enumerate}
\item $(M,\lambda)$ has a cofinal orbit.
\item $W_\lambda[[M]]$ satisfies ${\rm AC}_{\rm WO}$.
\end{enumerate}
\end{theorem}

\begin{remark}
In each of Theorems \ref{thm:Zapletal1}, \ref{thm:Zapletal2}, and \ref{thm:Zapletal3}, the definably closed hypothesis is needed only for (2)$\Rightarrow$(1).
\end{remark}

In particular, when combined with Theorem \ref{thm:main-stability-thm},
model-theoretic stability provides a criterion for choice principles not only in atomic sheaf toposes but also in permutation models.

\begin{cor}\label{cor:AC-criterion-main}
Let $\lambda,\kappa$ be infinite cardinals, and suppose that a structure-class $\K_\lambda$ satisfies the coherence axiom, the strong AP, and has a monster structure $M_\K\in\K$.
Then the following conditions are equivalent.
%
%
\begin{enumerate}
\item ${\rm Sh}(\K_\lambda^{\rm op},J_{\rm at})$ satisfies the axiom of choice for (external) $\kappa$-families.
\item Mejak's AC-criterion: Every $\kappa$-fan $\{A\embed B_\alpha\}_{\alpha<\kappa}$ admits a cocone in $\K_\lambda$. 
\item For every $\kappa$-fan $\{A\embed B_\alpha\}_{\alpha<\kappa}$ in $\K_\lambda$, in the topos of continuous ${\rm Aut}(M_\K/A)$-actions, the product 
$\displaystyle{\prod_{\alpha<\kappa}{\rm Aut}(M_\K/A)/{\rm Aut}(M_\K/B_\alpha)}$ is nonempty.
\item Zapletal's AC-criterion: $(M_\K,\lambda)$ is dynamically $\kappa$-complete.
\item The permutation model $W_\lambda[[M_\K]]$ satisfies the axiom of choice for (internal) $\kappa$-families.
\end{enumerate}

If $\kappa$ and $\lambda$ and $\K_\lambda$ satisfies the $\lambda$-TV, then conditions (6)--(11) are also equivalent to conditions (1)--(5).
\begin{enumerate}\setcounter{enumi}{5}
\item $\K$ is Galois $\lambda$-stable.
\item ${\rm Sh}(\K_\lambda^{\rm op},J_{\rm at})$ satisfies the axiom of choice for (external) $\lambda^+$-families.
\item ${\rm Sh}(\K_\lambda^{\rm op},J_{\rm at})$ satisfies the axiom of choice for (external) set-indexed families.
\item $\K_\lambda$ has the universal extension property.
\item $(M_\K,\lambda)$ has a cofinal orbit.
\item The permutation model $W_\lambda[[M_\K]]$ satisfies the axiom of choice for (internal) well-orders.
\end{enumerate}
\end{cor}

\begin{cor}\label{cor:DC-criterion-main}
Let $\lambda$ be an infinite cardinal, and suppose that a structure-class $\K_\lambda$ satisfies the coherence axiom, the strong AP, and has a monster structure $M_\K\in\K$.
Then the following conditions are equivalent.
\begin{enumerate}
\item ${\rm Sh}(\K_\lambda^{\rm op},J_{\rm at})$ satisfies the axiom of dependent choice.
\item Mejak's DC-criterion: Every $\om$-chain $\{A_n\embed A_{n+1}\}_{n\in\om}$ admits a cocone in $\K_\lambda$. 
\item For every $\om$-chain $\{A_n\embed A_{n+1}\}_{n\in\om}$ in $\K_\lambda$, in the topos of continuous ${\rm Aut}(M_\K/A_0)$-actions, the inverse limit, 
$\displaystyle{\lim_{n\in\om^{\rm op}}{\rm Aut}(M_\K/A_0)/{\rm Aut}(M_\K/A_n)}$, is nonempty.
\item Zapletal's DC-criterion: $(M_\K,\lambda)$ is DC-complete.
\item The permutation model $W_\lambda[[M_\K]]$ satisfies the axiom of dependent choice.
\end{enumerate}
\end{cor}

Of course, since Zapletal's framework \cite{Zap26} is broad, extending beyond automorphism group actions, finding an extension that encompasses his framework is an interesting problem.

\begin{ack}
The author is very grateful to Alex Simpson for valuable discussions.
The author's research was partially supported by JSPS KAKENHI Grant Number 23K28036 and 26K06892, and the JSPS-MESI Bilateral Joint Research Project 120265001.
\end{ack}

\medskip
\noindent
{\bf AI Disclosure.}
This research originated when the author decided to introduce Mejak's atomic topos models (which separate internal CC, DC, and AC \cite{Mej19}) in a textbook currently under preparation.
Since the author wanted to use the simplest possible examples for the textbook, he asked ChatGPT to suggest simple atomic toposes separating CC/DC/AC.
This led to the accidental start of the research assisted by ChatGPT.

ChatGPT, GPT-5.6 Sol, was used extensively during exploratory proof development.
Additionally, Claude Opus 5, alongside ChatGPT, was employed during multiple rounds of local peer review prior to the upload of this article.
The author guided the direction of the research, verified all mathematical arguments, refined and streamlined the proofs, and wrote the final text for all of them.
The author assumes full responsibility for the content.




\bibliographystyle{plain}
\bibliography{atomic-references}
\end{document}